\documentclass[11pt]{amsart}
\usepackage{amssymb}
\usepackage{amscd}
\usepackage{amsmath}
\usepackage[all]{xy}
\usepackage{mathrsfs}
 \usepackage{color}

\usepackage{mathtools}
\mathtoolsset{showonlyrefs=true}

\newcommand{\RG}{\rgamma}
\newcommand{\f}{f}
\newcommand{\bT}{\mathbb{T}}
\newcommand{\PP}{\mathcal{P}}

\DeclareMathOperator{\rank}{rank}

\DeclareMathOperator{\Frob}{Fr}
\DeclareMathOperator{\ES}{ES}
\DeclareMathOperator{\KS}{KS}
\DeclareMathOperator{\StS}{SS}
\DeclareMathOperator{\GL}{GL}
\DeclareMathOperator{\SL}{SL}
\DeclareMathOperator{\tr}{tr}
\DeclareMathOperator{\prim}{prime}

\DeclareMathOperator{\Ker}{Ker}
\DeclareMathOperator{\can}{can}

\calclayout

\theoremstyle{plain}
\newtheorem{theorem}{Theorem}[section]
\newtheorem{theorem*}{Theorem}
\newtheorem{proposition}[theorem]{Proposition}
\newtheorem{lemma}[theorem]{Lemma}

\newtheorem{conjecture}[theorem]{Conjecture}

\theoremstyle{remark}
\newtheorem{remark}[theorem]{Remark}

\theoremstyle{definition}
\newtheorem{definition}[theorem]{Definition}

\newtheorem{hypothesis}[theorem]{Hypothesis}

\newtheorem*{acknowledgments}{Acknowledgements}

\font\russ=wncyr10  1
\def\sha{\hbox{\russ\char88}}

\DeclareMathOperator{\Aut}{Aut}

\DeclareMathOperator{\Gal}{Gal}

\DeclareMathOperator{\Hom}{Hom}

\DeclareMathOperator{\N}{N}

\DeclareMathOperator{\im}{im}

\newcommand{\bF}{\mathbb{F}}
\newcommand{\HH}{\mathbb{H}}

\newcommand{\TT}{\mathbb{T}}
\newcommand{\WW}{\mathbb{W}}

\newcommand{\QQ}{\mathbb{Q}}
\newcommand{\cA}{\mathcal{A}}

\newcommand{\cD}{\mathcal{D}}

\newcommand{\cF}{\mathcal{F}}
\newcommand{\cG}{\mathcal{G}}

\newcommand{\cK}{\mathcal{K}}

\newcommand{\cN}{\mathcal{N}}
\newcommand{\cO}{\mathcal{O}}
\newcommand{\cP}{\mathcal{P}}
\newcommand{\cQ}{\mathcal{Q}}

\newcommand{\fq}{\mathfrak{q}}
\newcommand{\fp}{\mathfrak{p}}

\newcommand{\fn}{\mathfrak{n}}

\newcommand{\fz}{\mathfrak{z}}

\newcommand{\CC}{\mathbb{C}}

\newcommand{\GG}{\mathbb{G}}

\newcommand{\Q}{\mathbb{Q}}
\newcommand{\RR}{\mathbb{R}}

\newcommand{\ZZ}{\mathbb{Z}}
\newcommand{\Z}{\mathbb{Z}}

\newcommand{\Reg}{\mathrm{Reg}}

\newcommand{\bz}{\mathbb{Z}}

\newcommand{\ord}{\mathrm{ord}}

\newcommand{\rgamma}{\mathbf{R}\Gamma}
\newcommand{\rhom}{\mathbf{R}\Hom}
\newcommand{\lotimes}{\otimes^{\mathbf{L}}}

\makeatletter
 
  \@addtoreset{equation}{subsection}
\makeatother

\begin{document}

\title[]{On Euler systems for motives and Heegner points}

\author{Takenori Kataoka and Takamichi Sano}

\begin{abstract}
We formulate an Iwasawa main conjecture for a higher rank Euler system for a general motive. We prove ``one half" of the main conjecture under mild hypotheses. We also formulate a conjecture on ``Darmon-type derivatives" of Euler systems and give an application to the Tamagawa number conjecture. Lastly, we specialize our general framework to the setting of Heegner points and give a natural interpretation of the Heegner point main conjecture in terms of rank two Euler systems. 
\end{abstract}

\address{Keio University,
Department of Mathematics,
3-14-1 Hiyoshi\\Kohoku-ku\\Yokohama\\223-8522,
Japan}
\email{tkataoka@math.keio.ac.jp}

\address{Osaka City University,
Department of Mathematics,
3-3-138 Sugimoto\\Sumiyoshi-ku\\Osaka\\558-8585,
Japan}
\email{sano@osaka-cu.ac.jp}

\maketitle

\tableofcontents

\section{Introduction}

The aim of this paper is to study Iwasawa theory for motives in terms of higher rank Euler systems. 

Let $K$ be a number field and $M$ a (pure) motive defined over $K$. Let $p$ be an odd prime number and $T$ a stable lattice of the $p$-adic \'etale realization of $M$. For simplicity, we assume that the coefficient ring $\cA$ of $T$ is the ring of integers of a finite extension of $\QQ_p$ (e.g., $\cA=\ZZ_p$). Let $T^\ast(1)$ denote the Kummer dual of $T$. Then we define the {\it basic rank} of $T$ by
$$r=r_T := {\rm rank}_\cA \left( \bigoplus_{v\in S_\infty(K)} H^0(K_v,T^\ast(1))\right),$$
where $S_\infty(K)$ denotes the set of all infinite places of $K$. 

One expects that there is a canonical Euler system for $T$ of rank $r$. For example, when $T=\ZZ_p(1)$, one sees that $r_{\ZZ_p(1)}=1$ if and only if $K$ is either $\QQ$ or an imaginary quadratic field. When $K=\QQ$, we have the cyclotomic unit Euler system. When $K$ is an imaginary quadratic field, we have the elliptic unit Euler system. If $r_{\ZZ_p(1)}>1$, then it is known that conjectural Rubin-Stark elements constitute an Euler system. When $T$ is the $p$-adic Tate module of an elliptic curve over $\QQ$ and $K = \QQ$, then $r_T=1$ and we have Kato's Euler system. In general, the validity of the equivariant Tamagawa number conjecture for the dual motive $M^\ast(1)$ implies the existence of a canonical higher rank Euler system for $T$ of rank $r$ which is related to leading terms of (complex) $L$-functions for $M^\ast(1)$ at $s=0$. 

In this paper, we study arithmetic properties of Euler systems in a general setting and give a new example of our general theory in the setting of Heegner points. 

\subsection{The Iwasawa main conjecture for motives}

We first formulate an Iwasawa main conjecture for $T$. There are two types of the formulation: ``with $p$-adic $L$-functions" and ``without $p$-adic $L$-functions". In this paper, we study only the latter. 

We sketch the formulation. Fix a finite set $S$ of places of $K$ which contains all infinite and $p$-adic places of $K$ and all ``bad" places for $T$. Let $L/K$ be a finite abelian extension and $K_\infty/K$ a $\ZZ_p^d$-extension for some $d\geq 1$. 
We suppose that, for each $v \in S_{\infty}(K)$ that ramifies in $L$, we have $\rank_{\cA}(H^0(K_v, T^\ast(1))) = \frac{1}{2} \rank_{\cA}(T)$.
We set $L_\infty:=L \cdot K_\infty$ and $\Lambda:=\cA[[\Gal(L_\infty/K)]]$. Let $\TT:=T\otimes \Lambda$ be the $\Lambda$-adic deformation of $T$. Then one can construct a canonical map
$$\Theta: {\det}_\Lambda^{-1}(\rgamma(\cO_{K,S},\TT)) \to {\bigcap}_\Lambda^r H^1(\cO_{K,S},\TT). $$
Here $\rgamma(\cO_{K,S},\TT)$ denotes the usual $S$-cohomology complex, $H^i(\cO_{K,S},\TT)$ its cohomology, and $\bigcap_\Lambda^r$ the $r$-th ``exterior power bidual" over $\Lambda$ (see \S \ref{notation} below). The Iwasawa main conjecture is formulated as follows. 

\begin{conjecture}[The Iwasawa main conjecture, see Conjecture \ref{IMC}]\label{IMC intro}
Suppose that a canonical rank $r(=r_T)$ Euler system $c_{L_\infty} \in {\bigcap}_\Lambda^r H^1(\cO_{K,S},\TT)$ is given. Then there exists a $\Lambda$-basis
$$\fz_{L_\infty} \in {\det}_\Lambda^{-1}(\rgamma(\cO_{K,S},\TT))$$
such that $\Theta(\fz_{L_\infty})=c_{L_\infty}$. 
\end{conjecture}

When $T=\ZZ_p(1)$, a formulation of this form is given by Burns, Kurihara and the second author in \cite[Conj.~3.1 and Th.~3.4]{bks2}. 

Let us consider the simplest ``non-equivariant" case, i.e., $L=K$ and $d=1$ (so that $L_\infty=K_\infty$ is a $\ZZ_p$-extension of $K$). Then one proves that Conjecture \ref{IMC intro} is equivalent to the following ``classical" formulation:
$${\rm char}_\Lambda\left(  {\bigcap}_\Lambda^r H^1(\cO_{K,S},\TT)/\Lambda \cdot c_{K_\infty}\right) = {\rm char}_\Lambda(H^2(\cO_{K,S},\TT)),$$
where ${\rm char}$ denotes the characteristic ideal.
(See Proposition \ref{imc equivalent}.)

One of the main results of this paper is to prove ``one half" of the Iwasawa main conjecture for any given Euler system under some standard hypotheses.

\begin{theorem}[see Theorem \ref{main} for the precise statement]\label{main intro}
Assume $r=r_T \geq 1$, $p\geq 5$, and some mild hypotheses. 
Then, for any rank $r$ Euler system $c$, 
there exists an element
$$\fz_{L_{\infty}} \in {\det}_\Lambda^{-1}(\rgamma(\cO_{K,S},\TT))$$
such that $\Theta(\fz_{L_\infty}) = c_{L_\infty}$. 

In particular, when $L=K$ and $d=1$, we have 
$${\rm char}_\Lambda\left(  {\bigcap}_\Lambda^r H^1(\cO_{K,S},\TT)/\Lambda \cdot c_{K_\infty}\right) \subset {\rm char}_\Lambda(H^2(\cO_{K,S},\TT)).$$
\end{theorem}

For the proof, we essentially use the theory of higher rank Euler, Kolyvagin, and Stark systems established by Burns, Sakamoto, and the second author \cite{bss}, and also an idea of the recent work by the first author in \cite{kataoka2}. 

\subsection{Derivatives of Euler systems}
Let $r=r_T$ be the basic rank of $T$ and $c$ an Euler system of rank $r$ for $T$. Fix a finite set $S$ of places of $K$ as above. Under suitable assumptions, one can consider a ``Darmon-type derivative" of $c$ for a $\ZZ_p$-extension $K_\infty/K$:
$$\kappa_\infty \in {\bigwedge}_\cA^r H^1(\cO_{K,S},T)\otimes_\cA I^e/I^{e+1},$$
where we set $e:={\rm rank}_\cA(H^2(\cO_{K,S},T))$ and $I:=\ker (\cA[[\Gal(K_\infty/K)]]\twoheadrightarrow \cA)$ (see \S \ref{der iwasawa section}). We formulate a conjecture which relates $\kappa_\infty$ with the leading term $L_S^\ast(M^\ast(1),0)$ of the $S$-truncated (complex) $L$-function for the dual motive $M^\ast(1)$ at $s=0$. To do this we introduce an ``(extended) special element" for $T$ over $K$
$$\widetilde \eta_K \in \CC_p\otimes_{\ZZ_p} {\bigwedge}_\cA^{r+e} H^1(\cO_{K,S},T),$$
which is by definition related to the leading term (see Definition \ref{def ext}). This element is a natural generalization of the ``Birch and Swinnerton-Dyer element" introduced by Burns, Kurihara, and the second author in \cite[Def.~2.4]{bks4}. We naturally construct a ``Bockstein regulator map"
$${\rm Boc}_\infty: \CC_p\otimes_{\ZZ_p} {\bigwedge}_\cA^{r+e} H^1(\cO_{K,S},T) \to \CC_p\otimes_{\ZZ_p} {\bigwedge}_\cA^r H^1(\cO_{K,S},T)\otimes_\cA I^e/I^{e+1},$$
and formulate the conjecture as follows. 

\begin{conjecture}[Conjecture \ref{der iw}]\label{derivative intro}
We have
$$\kappa_\infty=(-1)^{re}{\rm Boc}_\infty(\widetilde \eta_K) \text{ in }\CC_p\otimes_{\ZZ_p}{\bigwedge}_{\cA}^r H^1(\cO_{K,S},T)\otimes_\cA I^e/I^{e+1}.$$
\end{conjecture}

This conjecture generalizes the ``generalized Perrin-Riou conjecture" in  \cite[Conj.~4.9]{bks4} and also the ``Iwasawa theoretic Mazur-Rubin-Sano conjecture" in \cite[Conj.~4.2]{bks2}. 

We give a strategy for proving the Tamagawa number conjecture for $M^\ast(1)$ by using the Iwasawa main conjecture and Conjecture \ref{derivative intro} (see Theorem \ref{descent}). This result is a generalization of \cite[Th.~7.6]{bks4}. 

\subsection{Heegner points}

To give a new example of our general theory, we study Heegner points in detail. Let $E$ be an elliptic curve over $\QQ$ and $K$ an imaginary quadratic field satisfying the Heegner hypothesis for $E$ (i.e., every prime divisor of the conductor of $E$ splits in $K$). We consider the motive $M=h^1(E/K)(1)$. In this case, $T$ is the $p$-adic Tate module of $E$ and the coefficient ring is $\cA=\ZZ_p$. 

The Euler system of Heegner points is usually considered to be a rank one Euler system. However, it is known that it does not satisfy the natural definition of Euler systems given by Rubin \cite{R}. In this sense, it might be unnatural to regard the Heegner point Euler system as a rank one Euler system. 

In this paper, we make the following observation: {\it it is natural to interpret the system of Heegner points as a rank two Euler system}. In fact, {\it the basic rank $r_T$ is two} in this setting, since we have
$$\bigoplus_{v\in S_\infty(K)} H^0(K_v,T^\ast(1)) = H^0(\CC, T^\ast(1))= T^\ast(1)$$
and this is a free $\ZZ_p$-module of rank two. 

Our idea of interpreting Heegner points as a rank two Euler system is as follows. We assume that $E$ has good ordinary reduction at $p$. Take a finite set $S$ of places of $K$ as usual. Let $K_\infty/K$ be the anticyclotomic $\ZZ_p$-extension and set $\Lambda:=\ZZ_p[[\Gal(K_\infty/K)]]$ and $\TT:=T\otimes \Lambda$. In the Selmer group ${\rm Sel}(\TT)={\rm Sel}(K,\TT)$, one has a $\Lambda$-adic Heegner point $y_\infty \in {\rm Sel}(\TT)$ (see \cite[\S 3.1]{castella} for example). Under some mild conditions, we construct an isomorphism
$$Q(\Lambda) \otimes_\Lambda {\bigcap}_\Lambda^2 H^1(\cO_{K,S},\TT) \simeq Q(\Lambda)\otimes_\Lambda \left({\rm Sel}(\TT)\otimes_\Lambda {\rm Sel}(\TT)^\iota\right),$$
where $Q(\Lambda)$ denotes the quotient field of $\Lambda$ and $(-)^\iota$ the module on which $\Lambda$ acts via the natural involution. We define a ``Heegner element"
$$z_\infty^{\rm Hg} \in Q(\Lambda)\otimes_\Lambda {\bigcap}_\Lambda^2 H^1(\cO_{K,S},\TT)$$
to be the element corresponding to 
$$y_\infty\otimes y_\infty \in {\rm Sel}(\TT)\otimes_\Lambda {\rm Sel}(\TT)^\iota$$
under the isomorphism above (see Definition \ref{def lambda}). 

We show that the ``Heegner point main conjecture" of Perrin-Riou is equivalent to our formulation of the Iwasawa main conjecture for $z_\infty^{\rm Hg}$. 

\begin{theorem}[Theorem \ref{heeg equivalent}]
The Heegner point main conjecture holds if and only if we have $z_\infty^{\rm Hg} \in {\bigcap}_\Lambda^2 H^1(\cO_{K,S},\TT)$ and an equality
$${\rm char}_\Lambda \left( {\bigcap}_\Lambda^2 H^1(\cO_{K,S},\TT)/\Lambda \cdot z_\infty^{\rm Hg}\right) = {\rm char}_\Lambda(H^2(\cO_{K,S},\TT)).$$
\end{theorem}

As an application, we give a formal construction of a rank two Euler system whose ``$K_\infty$-component" is $z_\infty^{\rm Hg}$. 

\begin{theorem}[Theorem \ref{heeg iwasawa}]\label{euler intro}
Assume the Heegner point main conjecture. Then there exists a rank two Euler system $c$ such that $c_{K_\infty}=z_\infty^{\rm Hg}$. 
\end{theorem}

We remark that our construction is non-canonical: roughly speaking, the constructed Euler system is just a ``lift" of $z_\infty^{\rm Hg}$. (The Heegner point main conjecture is assumed in order to ensure the existence of a lift.) However, this result at least gives evidence for the existence of a canonical rank two Euler system which is related to Heegner points. 

We also remark that there is a non-Iwasawa theoretic version of Theorem \ref{euler intro}: see Theorem \ref{heeg bottom}. A Heegner element (over $K$) is more explicitly defined in this case without assuming that $E$ has good ordinary reduction at $p$ (see Definition \ref{def heeg}). 

Lastly, we give an explicit interpretation of Conjecture \ref{derivative intro} for the Heegner element $z_\infty^{\rm Hg}$. We assume the following:
\begin{itemize}
\item[(i)] $E(K)[p]=0$;
\item[(ii)] $ {\rm rank}(E(\QQ))\geq 1 $ and ${\rm rank}(E^K(\QQ)) \geq 1$, where $E^K$ denotes the quadratic twist of $E$ by $K$;
\item[(iii)] $\# \sha(E/K)[p^\infty]<\infty$.
\end{itemize}
(See Hypothesis \ref{hypk}.) Note that the Heegner hypothesis and the validity of the parity conjecture imply that ${\rm rank}(E(K))$ is odd. So the condition (ii) implies ${\rm rank}(E(K))\geq 3$. (One sees that Conjecture \ref{derivative intro} is not interesting when ${\rm rank}(E(K))=1$.)

We set $e:={\rm rank}_{\ZZ_p}(H^2(\cO_{K,S},T))$ and $I:=\ker (\Lambda \twoheadrightarrow \ZZ_p)$. Let
$$\kappa_\infty^{\rm Hg} \in {\bigwedge}_{\ZZ_p}^2 H^1(\cO_{K,S},T)\otimes_{\ZZ_p}I^e/I^{e+1}$$
be the Darmon-type derivative of $z_\infty^{\rm Hg}$. 
We define a canonical ``anticyclotomic Bockstein regulator"
$$R_{K_\infty}^{\rm Boc} \in \QQ_p\otimes_{\ZZ_p}{\bigwedge}_{\ZZ_p}^2 H^1(\cO_{K,S},T)\otimes_{\ZZ_p} I^e/I^{e+1}$$
as an analogue of the ``(cyclotomic) Bockstein regulator" introduced by Burns, Kurihara, and the second author in \cite[Def.~4.11]{bks4}. 

We prove that Conjecture \ref{derivative intro} in this case is equivalent to the following explicit formula. 

\begin{conjecture}[see Proposition \ref{explicit derivative}]\label{explicit intro}
We have
$$\kappa_\infty^{\rm Hg}=\frac{L_S^\ast(E/K,1)\sqrt{|D_K|}}{\Omega_{E/K}\cdot R_{E/K}}\cdot R_{K_\infty}^{\rm Boc} ,$$
where $L_S^\ast(E/K,1)$ denotes the leading term of the $S$-truncated $L$-function of $E/K$ at $s=1$, $D_K$ the discriminant of $K$, $\Omega_{E/K}$ the N\'eron period, and $R_{E/K}$ the N\'eron-Tate regulator. 
\end{conjecture}

According to the conjectural Birch-Swinnerton-Dyer formula, the analytic constant 
$$\frac{L_S^\ast(E/K,1)\sqrt{|D_K|}}{\Omega_{E/K}\cdot R_{E/K}}$$
should be equal to the algebraic constant
$${\rm Eul}_S\cdot \#\sha(E/K)[p^\infty]\cdot {\rm Tam}(E/K)$$
up to $\ZZ_p^\times$, where ${\rm Eul}_S$ denotes the product of Euler factors at primes in $S$ (so that ${\rm Eul}_S\cdot L^\ast(E/K,1)=L_S^\ast(E/K,1)$) and ${\rm Tam}(E/K)$ the product of Tamagawa factors of $E/K$. 
We prove that an algebraic variant of Conjecture \ref{explicit intro} follows from the Heegner point main conjecture up to $\ZZ_p^\times$. 

\begin{theorem}[Theorem \ref{acbsd}]\label{algebraic intro}
Assume the Heegner point main conjecture. Then there exists $u \in \ZZ_p^\times$ such that
$$\kappa_\infty^{\rm Hg} =u\cdot  {\rm Eul}_S\cdot \#\sha(E/K)[p^\infty]\cdot {\rm Tam}(E/K)\cdot R_{K_\infty}^{\rm Boc}.$$
\end{theorem}

In a forthcoming work, we show that Conjecture \ref{explicit intro} (or rather its algebraic variant) implies the conjecture of Bertolini and Darmon \cite[Conj.~4.5(1)]{BD} (see also \cite[Conj.~3.6]{AC}). This gives further evidence for Conjecture \ref{explicit intro}. 

Finally, we give a strategy for proving the Birch-Swinnerton-Dyer formula for $E/K$.

\begin{theorem}[Theorem \ref{heegdescent}]\label{heegdescent intro}
If we assume
\begin{itemize}
\item the Heegner point main conjecture,
\item Conjecture \ref{explicit intro}, and 
\item $R_{K_\infty}^{\rm Boc}\neq 0$, 
\end{itemize}
then the $p$-part of the Birch-Swinnerton-Dyer formula for $E/K$ holds, i.e., there exists $u \in \ZZ_p^\times$ such that
$$L^\ast(E/K,1)=u\cdot \# \sha(E/K)[p^\infty]\cdot {\rm Tam}(E/K)\cdot \frac{1}{\sqrt{|D_K|}} \Omega_{E/K}\cdot R_{E/K}.$$
\end{theorem}

We remark that Theorems \ref{algebraic intro} and \ref{heegdescent intro} are analogues of the results of Burns, Kurihara, and the second author \cite[Th.~7.3 and 7.6]{bks4} respectively. 

\subsection{Notation}\label{notation}

For a commutative ring $R$ and an $R$-module $N$, we set $N^\ast:=\Hom_R(N,R)$. The $R$-torsion submodule of $N$ is denoted by $N_{\rm tors}$. The $R$-torsion-free quotient $N/N_{\rm tors}$ is denoted by $N_{\rm tf}$. For a non-negative integer $a$, the $a$-th exterior power bidual of $N$ over $R$ is defined by
$${\bigcap}_R^a N :=\left({\bigwedge}_R^a (N^\ast)\right)^\ast. $$
For basic properties, see \cite[Appendix A]{sbA}. 

For an abelian group ($\ZZ$-module) $A$ and a prime number $p$, we set
$$A[p]:=\{a \in A \mid p\cdot a =0\}\text{ and }A[p^\infty]:=\{a\in A \mid p^n \cdot a =0 \text{ for some $n$}\}.$$

Let $K$ be a number field, which is regarded as a finite extension of $\QQ$ inside a fixed algebraic closure $\overline \QQ$ of $\QQ$. We denote the absolute Galois group $\Gal(\overline \QQ/K)$ by $G_K$. For each place $v$ of $K$, we fix a place $w$ of $\overline \QQ$ lying above $v$. The decomposition group of $w$ in $G_K$ is identified with $\Gal(\overline \QQ_w/K_v)$. In particular, we regard $\Gal(\overline \QQ_w/K_v)\subset G_K$. Let $K_v^{\rm ur}$ be the maximal unramified extension of $K_v$ inside $\overline \QQ_w$ and ${\rm Fr}_v \in \Gal(K_v^{\rm ur}/K_v)$ the Frobenius element of $v$. We fix a lift of ${\rm Fr}_v$ in $\Gal(\overline \QQ_w/K_v)$ and denote it also by the same symbol. 

The set of all infinite (resp.~$p$-adic) places of $K$ is denoted by $S_\infty(K)$ (resp.~$S_p(K)$). 

For a Galois extension $F/K$, we often denote $\Gal(F/K)$ by $\cG_F$. The set of finite places of $K$ which ramify in $F$ is denoted by $S_{\rm ram}(F/K)$. For a finite set $S$ of places of $K$, we set 
$$S_F:=\{w:\text{a place of $F$}\mid \text{the place of $K$ lying under $w$ belongs to $S$}\}$$
and 
$$S(F):=S\cup S_{\rm ram}(F/K).$$ 

We use some standard notations concerning Galois (\'etale) cohomology: $\rgamma(\cO_{K,S},-)$, $\rgamma_f(K_v,-)$, $\rgamma_{/f}(K_v,-)$, etc. For the definitions, see \cite[\S 1.4]{sbA} for example.

\section{Euler systems for motives}\label{sec euler}
In this section, we give a review of a general conjecture on Euler systems given in \cite[\S 4]{bss2} (see Conjecture \ref{conjes}). This conjecture predicts what kind of Euler system should exist in a general setting of motives. 
In \S \ref{sec ext}, we give a generalization of ``Birch-Swinnerton-Dyer elements" introduced in \cite[\S 2.2]{bks4}, which will be used in \S \ref{sec der}. 

\subsection{The definition of Euler systems}

Let $K$ be a number field and $p>2$ an odd prime number. Let $M$ be a (pure) motive defined over $K$ with coefficients in a finite dimensional semisimple commutative $\QQ$-algebra $R$, which is necessarily a finite product of number fields. Let $A$ be a finite extension of $\QQ_p$ which arises as a component of $\QQ_p\otimes_\QQ R$ and $\cA$ the ring of integers of $A$. Let $V_p(M)$ be the $p$-adic \'etale realization of $M$ and set $V:= A \otimes_{\QQ_p\otimes_\QQ R} V_p(M)$, which is a finite dimensional $A$-vector space endowed with a continuous action of $G_K$. Fix a $G_K$-stable lattice $T \subset V$, which is a free $\cA$-module of finite rank. Let $T^\ast(1):=\Hom_{\cA}(T,\cA(1))$ be the Kummer dual of $T$. We set
$$Y_K(T^\ast(1)):=\bigoplus_{v\in S_\infty(K)}H^0(K_v,T^\ast(1)).$$
Note that, for each $v \in S_{\infty}(K)$, the $\cA$-module $H^0(K_v,T^\ast(1))$ is a direct summand of $T^\ast(1)$, so in particular it is free of rank less than or equal to ${\rm rank}_{\cA}(T^\ast(1)) = {\rm rank}_{\cA}(T)$. 
Therefore, $Y_K(T^\ast(1))$ is also a free $\cA$-module with an upper bound of the rank.





\begin{definition}\label{def basic}
We define the {\it basic rank} of $T$ by
$$r=r_T:={\rm rank}_\cA(Y_K(T^\ast(1))).$$
\end{definition}

We fix a finite set $S$ of places of $K$ such that
$$S_\infty(K) \cup S_p(K)\cup S_{\rm ram}(T) \subset S,$$
where $S_{\rm ram}(T)$ denotes the set of finite places of $K$ at which $T$ ramifies. For any $v \notin S$, we set
$$P_v(x)=P_v(x;T):=\det(1-{\rm Fr}_v^{-1}x \mid T^\ast(1)) \in \cA[x].$$

Let $\cK/K$ be an abelian extension.
Let $\Omega(\cK)$ be the set of finite subextensions $F/K$ of $\cK/K$. 
For each $F \in \Omega(\cK)$ we set
$$\cG_F:=\Gal(F/K)$$
and 
$$S(F):=S\cup S_{\rm ram}(F/K).$$

We impose the following hypothesis on the extension $\cK/K$:
for each $F \in \Omega(\cK)$, the $\cA[\cG_F]$-module
$$Y_F(T^\ast(1)):=\bigoplus_{w \in S_\infty(F)}H^0(F_w,T^\ast(1))$$
is free of rank $r_T$. 
One sees that this condition is equivalent to the following:
$$\text{For any $v \in S_\infty(K)$ which ramifies in $\cK$, we have ${\rm rank}_\cA(H^0(K_v,T^\ast(1))) = \frac 12 {\rm rank}_\cA(T)$.}$$
For example, this hypothesis is satisfied as long as every $v \in S_\infty(K)$ splits completely in $\cK$.

Let $\Sigma$ be a finite set (possibly empty) of places of $K$ which is disjoint from $S(F)$ for any $F \in \Omega(\cK)$. (This means that $\Sigma$ is disjoint from $S$ and every $v\in \Sigma$ is unramified in $\cK$.) Following \cite[\S 2.3]{sbA}, for any $F \in \Omega(\cK)$, we define the $\Sigma$-modified cohomology complex $\rgamma_\Sigma(\cO_{F,S(F)},T)$ by the exact triangle
$$\rgamma_\Sigma(\cO_{F,S(F)},T) \to \rgamma(\cO_{F,S(F)},T) \to \bigoplus_{w\in \Sigma_F}\rgamma_f(F_w,T)\to .$$

For any $F, F' \in \Omega(\cK)$ with $F \subset F'$, we write
$${\rm Cor}_{F'/F}^{r}: {\bigcap}_{\cA[\cG_{F'}]}^{r} H_\Sigma^1(\cO_{F',S(F')},T) \to {\bigcap}_{\cA[\cG_F]}^{r} H_\Sigma^1(\cO_{F,S(F')},T)$$
for the map induced by the corestriction map ${\rm Cor}_{F'/F}: H_\Sigma^1(\cO_{F',S(F')},T) \to H^1_\Sigma(\cO_{F,S(F')},T)$. 

\begin{definition}\label{defn:ES}
An {\it Euler system} of rank $r$ for $(T,\cK)$ (with an implicit choice of $S$ and $\Sigma$) is an element
$$c=(c_F)_F \in \prod_{F \in \Omega(\cK)} {\bigcap}_{\cA[\cG_F]}^{r} H_\Sigma^1(\cO_{F,S(F)},T)$$
satisfying the following: for any $F, F' \in \Omega(\cK)$ with $F\subset F'$, we have
$${\rm Cor}_{F'/F}^{r}(c_{F'})= \left( \prod_{v \in S(F')\setminus S(F)} P_v({\rm Fr}_v^{-1})\right) c_F.$$
The set ($\cA[[\Gal(\cK/K)]]$-module) of Euler systems of rank $r$ for $(T,\cK)$ is denoted by ${\rm ES}_r(T,\cK)$. 
\end{definition}

\subsection{Conjectural Euler systems}

In this subsection, we formulate an explicit conjecture concerning the existence of an Euler system for $T$ which is related with $L$-functions of the motive $M^\ast(1)$, under the following hypothesis. 

\begin{hypothesis}\label{hypint}
For any $F \in \Omega(\cK)$, we have
\begin{itemize}
\item[(i)] $H^0(F,T)=0$, 
\item[(ii)] either $\Sigma$ is non-empty or $H^1(\cO_{F,S(F)},T)$ is $\cA$-free, and
\item[(iii)] for any $w \in \Sigma_F$, we have $H^0(F_w,T)=0$. 
\end{itemize}
\end{hypothesis}


\begin{remark}
If $H^0(F,T)=0$ and $\Sigma$ is non-empty, then one easily sees that $H^1_{\Sigma}(\cO_{F,S(F)},T)$ is $\cA$-free. So Hypothesis \ref{hypint}(ii) can be replaced by
\begin{itemize}
\item[(ii$'$)] $H^1_{\Sigma}(\cO_{F,S(F)},T)$ is $\cA$-free.
\end{itemize}
\end{remark}

\begin{remark}
Hypothesis \ref{hypint}(iii) implies that $H^1_f(F_w, T)$ is finite for any $w \in \Sigma_F$. So in this case we have
$$\QQ_p\otimes_{\ZZ_p} H^i_\Sigma(\cO_{F,S(F)},T) = H^i(\cO_{F,S(F)},V).$$
This identification will frequently be used. 
\end{remark}


We set some notations. Let $F \in \Omega(\cK)$. We write $\widehat \cG_F$ for the set of $\overline \QQ$-valued characters of $\cG_F$. By the fixed embeddings $\overline \QQ \hookrightarrow \CC$ and $\overline \QQ \hookrightarrow \CC_p$, each $\chi \in \widehat \cG_F$ is regarded as both $\CC$-valued and $\CC_p$-valued. For $\chi \in \widehat \cG_F$, we define the usual idempotent by  $e_\chi:=(\#\cG_F)^{-1}\sum_{\sigma \in \cG_F}\chi(\sigma)\sigma^{-1}$. 
We set
$$\Upsilon(T,F):=\{\chi \in \widehat \cG_F \mid e_\chi( \CC_p\otimes_{\ZZ_p}H^2(\cO_{F,S(F)},T))=0\}$$
and
$$e_{T,F}:=\sum_{\chi \in \Upsilon(T,F)}e_\chi \in A[\cG_F].$$
Roughly speaking, $e_{T, F}$ is the ``maximal" idempotent that annihilates $H^2(\cO_{F, S(F)}, T)$.

We recall the definition of motivic $L$-functions. To do this, we assume that the Euler factors have rational coefficients, i.e., $Q_v(x):=\det(1-{\rm Fr}_v^{-1}x \mid V_p(M^\ast(1)))$ belongs to $R[x]$ for any $v \notin S$. 

For $\chi \in \widehat \cG_F$, the $S(F)$-truncated $\chi$-twisted $L$-function of $M^\ast(1)$ is defined by
$$L_{S(F)}(M^\ast(1),\chi,s):=\prod_{v\notin S(F)} Q_v(\chi({\rm Fr}_v) {\N}v^{-s})^{-1}.$$ 
This is a complex function which takes values in $\CC \otimes_\QQ R$ and converges if ${\rm Re}(s)$ is large enough. We assume that it is analytically continued to the whole complex plane. 
We define a $\Sigma$-modified version by
$$L_{S(F),\Sigma}(M^\ast(1),\chi,s):=\left( \prod_{v\in \Sigma}Q_v(\chi({\rm Fr}_v){\N}v^{1-s})\right) L_{S(F)}(M^\ast(1),\chi,s).$$

We then define the $\cG_F$-equivariant $(S(F),\Sigma)$-modified $L$-function of $M^\ast(1)$ by
$$\theta_{F/K,S(F),\Sigma}(M^\ast(1),s):=\sum_{\chi \in \widehat \cG_F} L_{S(F),\Sigma}(M^\ast(1),\chi^{-1},s) e_\chi ,$$
which takes values in $\CC \otimes_\QQ R[\cG_F]$. 

We write $L_{S(F),\Sigma}^\ast(M^\ast(1),\chi,0) \in (\CC \otimes_\QQ R )^\times$ for the leading term of $L_{S(F),\Sigma}(M^\ast(1),\chi,s)$ at $s=0$, i.e., the leading coefficient in the Laurent expansion at $s=0$. We define the leading term of $\theta_{F/K,S(F),\Sigma}(M^\ast(1),s)$ by
$$\theta_{F/K,S(F),\Sigma}^\ast(M^\ast(1),0):=\sum_{\chi \in \widehat \cG_F} L_{S(F),\Sigma}^\ast(M^\ast(1),\chi^{-1},0)e_\chi \in (\RR \otimes_\QQ R[\cG_F])^\times.$$
We fix an embedding $\RR \hookrightarrow \CC_p$ and regard $\theta_{F/K,S(F),\Sigma}^\ast(M^\ast(1),0)$ as an element of $(\CC_p\otimes_\QQ R[\cG_F])^\times$. By the natural projection $\QQ_p \otimes_\QQ R\twoheadrightarrow A$, we regard $\theta_{F/K,S(F),\Sigma}^\ast(M^\ast(1),0)$ as an element of $(\CC_p\otimes_{\QQ_p}A[\cG_F])^\times=(\CC_p\otimes_{\ZZ_p}\cA[\cG_F])^\times$. 

In \S \ref{def per} below, we will define a canonical ``period-regulator isomorphism''
$$\lambda_{T,F}: e_{T,F}\left(\CC_p \otimes_{\ZZ_p} {\bigwedge}_{\cA[\cG_F]}^r H^1(\cO_{F,S(F)},T) \right)\xrightarrow{\sim} e_{T,F}\left( \CC_p \otimes_{\ZZ_p} {\bigwedge}_{\cA[\cG_F]}^r Y_F(T^\ast(1))^\ast \right),$$
where $Y_F(T^\ast(1))^\ast:=\Hom_{\cA}(Y_F(T^\ast(1)),\cA)$. 



\begin{conjecture}\label{conjes}
Assume Hypothesis \ref{hypint}, and fix an $\cA[[\Gal(\cK/K)]]$-basis 
$$b=(b_F)_F \in \varprojlim_{F \in \Omega(\cK)} {\bigwedge}_{\cA[\cG_F]}^r Y_F(T^\ast(1))^\ast (\simeq \varprojlim_{F\in \Omega(\cK)}\cA[\cG_F]=\cA[[\Gal(\cK/K)]]).$$ 
Then there exists a unique Euler system
$$c=c(b) \in {\rm ES}_r(T,\cK)$$
satisfying the following properties.
\begin{itemize}
\item[(i)] For every $F \in \Omega(\cK)$, we have
$$(1-e_{T,F}) c_F = 0,$$
i.e., $c_F \in {\bigcap}_{\cA[\cG_F]}^r H^1_\Sigma(\cO_{F,S(F)},T) \subset {\bigwedge}_{A[\cG_F]}^r H^1(\cO_{F,S(F)},V)$ belongs to $e_{T,F}{\bigwedge}_{A[\cG_F]}^r H^1(\cO_{F,S(F)},V)$. 
\item[(ii)] For every $F \in \Omega(\cK)$, we have
$$\lambda_{T,F}(c_F)=e_{T,F}\cdot \theta_{F/K,S(F),\Sigma}^\ast(M^\ast(1),0)\cdot b_F \text{ in }e_{T,F}\left(\CC_p \otimes_{\ZZ_p} {\bigwedge}_{\cA[\cG_{F}]}^r Y_{F}(T^\ast(1))^\ast   \right).$$
\end{itemize}
\end{conjecture}

It is convenient to give the following definition. 

\begin{definition}\label{special}
We fix $b$ as in Conjecture \ref{conjes}. For each $F \in \Omega(\cK)$, we define the {\it special element} for $T$ by
\begin{eqnarray*}
\eta_F=\eta_{F/K,S(F),\Sigma}(T)&:=&\lambda_{T,F}^{-1} \left(e_{T,F}\cdot \theta_{F/K,S(F),\Sigma}^\ast(M^\ast(1),0)\cdot b_F \right) \\
&\in& e_{T,F}\left(\CC_p \otimes_{\ZZ_p} {\bigwedge}_{\cA[\cG_F]}^r H^1(\cO_{F,S(F)},T) \right).
\end{eqnarray*}
(This is called the ``Bloch-Kato element" in \cite[Def.~4.10]{bss2}.) 
\end{definition}

\begin{remark}
One can show that the collection $(\eta_F)_{F \in \Omega(\cK)}$ satisfies the norm relation, i.e., we have
$${\rm Cor}_{F'/F}^{r}(\eta_{F'})= \left( \prod_{v \in S(F')\setminus S(F)} P_v({\rm Fr}_v^{-1})\right) \eta_F$$
for any $F ,F' \in \Omega(\cK)$ with $F \subset F'$. Also, $\eta_F$ satisfies the properties (i) and (ii) in Conjecture \ref{conjes} by definition. Thus Conjecture \ref{conjes} is equivalent to the following assertion: for every $F \in \Omega(\cK)$, we have
\begin{equation}\label{special int}
\eta_F \in {\bigcap}_{\cA[\cG_F]}^r H^1_\Sigma(\cO_{F,S(F)},T).
\end{equation}
The conjecture of this form is given in \cite[Conj.~4.15]{bss2}. 
\end{remark}

\begin{remark}\label{rem etnc}
By \cite[Rem.~2.11 and Th.~2.18]{sbA}, one sees that the equivariant Tamagawa number conjecture for $(M^\ast(1) \otimes_{K} F, \cA[\cG_F])$ (see \cite[Conj.~4]{BFetnc}) implies \eqref{special int} for $F \in \Omega(\cK)$. This gives theoretical evidence for Conjecture \ref{conjes}. 
\end{remark}
\begin{remark}\label{rem strong}
In this paper, we often assume the existence of a canonical Euler system and study its properties. So it may be reasonable to assume Conjecture \ref{conjes} throughout. However, it turns out that {\it assuming Conjecture \ref{conjes} is too strong}. For example, in the case of elliptic curves over $\QQ$, we have a canonical Euler system called Kato's Euler system, but {\it it is still not known whether it satisfies the properties in Conjecture \ref{conjes}}. In fact, Conjecture \ref{conjes} for Kato's Euler system is equivalent to a natural equivariant refinement of Perrin-Riou's conjecture, which has not yet been fully proved (see \S \ref{kato ex} below for the details). For this reason, we will only assume the existence of a canonical Euler system and propose conjectures for it. 
\end{remark}

\subsection{Examples}\label{sec ex}

Let us consider two special cases for which the conjectural Euler systems are more familiar.
The first case is the (conjectural) system of Rubin-Stark elements over number fields, which specializes to the cyclotomic unit Euler system when $K=\Q$ and to the elliptic unit Euler system when $K$ is an imaginary quadratic field.
The second is Kato's Euler system for elliptic curves over $\Q$.

\subsubsection{The Rubin-Stark Euler system}\label{rs ex}
Consider the ``$\GG_m$ case", i.e., $M=h^0(K)(1)$, $R=\QQ$, $A=\QQ_p$, $\cA=\ZZ_p$, and $T=\ZZ_p(1)$. Take $\cK$ so that every $v \in S_\infty(K)$ splits completely in $\cK$ and also choose $\Sigma$ so that Hypothesis \ref{hypint} is satisfied. In this case, we have
$$r=r_{\ZZ_p(1)}=\# S_\infty(K).$$
Also, for $F \in \Omega(\cK)$ and $\chi \in \widehat \cG_F$ we have
\begin{eqnarray*}
L_{S(F),\Sigma}(M^\ast(1),\chi,s)&=&L_{S(F),\Sigma}(\chi,s)\\
&:=&\prod_{v\in \Sigma}(1-\chi({\rm Fr}_v){\N}v^{1-s}) \prod_{v \notin S(F)} (1-\chi({\rm Fr}_v){\N}v^{-s})^{-1}.
\end{eqnarray*}
This is the usual $(S(F),\Sigma)$-modified Artin $L$-function for $\chi$ (see \cite[\S 3.1]{bks1} for example). So we have
$$\theta_{F/K,S(F),\Sigma}(M^\ast(1),s)=\theta_{F/K,S(F),\Sigma}(s):=\sum_{\chi \in \widehat \cG_F} L_{S(F),\Sigma}(\chi^{-1},s)e_\chi.$$
For any finite set $U$ of places of $K$, we set
$$Y_{F,U}:=\bigoplus_{w \in U_F}\ZZ \cdot w \text{ and }X_{F,U}:=\left\{\sum_{w} a_w \cdot w \in Y_{F,U} \ \middle| \ \sum_w a_w=0\right\}.$$
Then $Y_F(T^\ast(1))^\ast =Y_F(\ZZ_p)^\ast$ is identified with $\ZZ_p \otimes_\ZZ Y_{F,S_\infty(K)}$. Let
$$\delta_F: \RR \otimes_{\ZZ} \cO_{F,S(F)}^\times  \xrightarrow{\sim} \RR \otimes_\ZZ X_{F,S(F)}; \ u \mapsto -\sum_{w \in S(F)_F} \log|u|_w \cdot w$$
be the Dirichlet regulator. The period-regulator isomorphism in this case is defined by
\begin{eqnarray*}
\lambda_{\ZZ_p(1),F}: e_{\ZZ_p(1),F}\left(\CC_p \otimes_{\ZZ_p} {\bigwedge}_{\ZZ_p[\cG_F]}^r H^1(\cO_{F,S(F)},\ZZ_p(1)) \right) &\simeq& e_{\ZZ_p(1),F}\left(\CC_p \otimes_{\ZZ} {\bigwedge}_{\ZZ[\cG_F]}^r \cO_{F,S(F)}^\times \right)\\
&\stackrel{\delta_F}{\simeq}&  e_{\ZZ_p(1),F}\left(\CC_p \otimes_{\ZZ} {\bigwedge}_{\ZZ[\cG_F]}^r X_{F,S(F)} \right)\\
& \simeq & e_{\ZZ_p(1),F}\left( \CC_p \otimes_{\ZZ} {\bigwedge}_{\ZZ[\cG_F]}^r Y_{F,S_\infty(K)} \right),
\end{eqnarray*}
where the first isomorphism is induced by the Kummer isomorphism $\ZZ_p\otimes_\ZZ \cO_{F,S(F)}^\times \simeq H^1(\cO_{F,S(F)},\ZZ_p(1))$, and the last isomorphism follows by the definition of $e_{\ZZ_p(1),F}$ and a canonical exact sequence
$$0\to H^2(\cO_{F,S(F)},\QQ_p(1)) \to \QQ_p\otimes_\ZZ X_{F,S(F)} \to \QQ_p\otimes_\ZZ Y_{F,S_\infty(K)}\to 0$$
(see \cite[(9.2.1.2)]{nekovar} for example). One can choose a $\ZZ[\cG_F]$-basis $b_F$ of ${\bigwedge}_{\ZZ[\cG_F]}^r Y_{F,S_\infty(K)}$ by fixing a labeling $S_\infty(K)=\{v_1,\ldots,v_r\}$ and a place $w_i$ of $F$ lying above each $v_i$. Namely, one sets $b_F:=w_1\wedge \cdots \wedge w_r$. The Rubin-Stark element (for $F/K,S(F),\Sigma,S_\infty(K)$) is defined by
$$\eta_F^{\rm RS}:=\eta_{F/K,S(F),\Sigma}^{S_\infty(K)} :=\lambda_{\ZZ_p(1),F}^{-1}(e_{\ZZ_p(1),F}\cdot \theta_{F/K,S(F),\Sigma}^\ast(0)\cdot b_F)  .$$
(See \cite[\S 5.1]{bks1} for example.) This coincides with the special element $\eta_{F/K,S(F),\Sigma}(\ZZ_p(1))$ in Definition \ref{special}. The conjecture \eqref{special int} is equivalent to the ($p$-part of the) Rubin-Stark conjecture (see \cite[Conj.~B$'$]{rubinstark} or \cite[Conj.~5.1]{bks1}). Thus, if we assume the Rubin-Stark conjecture for all $F \in \Omega(\cK)$, then the conjectural Euler system in Conjecture \ref{conjes} coincides with the Rubin-Stark Euler system
$$\eta^{\rm RS}:=(\eta_F^{\rm RS})_F \in {\rm ES}_r(\ZZ_p(1),\cK).$$
In particular, Conjecture \ref{conjes} is true when $K=\QQ$. (In this case, the Rubin-Stark Euler system is the cyclotomic unit Euler system.)

\subsubsection{Kato's Euler system}\label{kato ex}

Let $E$ be an elliptic curve over $\QQ$ and consider the case when $K=\QQ$, $M=h^1(E)(1)$, $R=\QQ$, $A=\QQ_p$, $\cA=\ZZ_p$, and $T$ is a lattice of $V_p(E)$. In this case, we have
$$r=r_T=1.$$
We take $\cK/\QQ$ to be an abelian $p$-extension, $\Sigma$ to be empty and assume Hypothesis \ref{hypint}. (If $T=T_p(E)$, then Hypothesis \ref{hypint} is equivalent to $E(\QQ)[p]=0$.) For each $F \in \Omega(\cK)$, Kato \cite{katoasterisque} constructed a ``zeta element"
$$z_F^{\rm Kato} \in H^1(\cO_{F,S(F)}, V_p(E)).$$
(See \cite[Def.~6.8]{bss2} for normalization. It depends on the choice of a $\ZZ_p[\cG_F]$-basis $b_F \in Y_F(T^\ast(1))^\ast$.) If we assume the ``integrality", i.e., $z_F^{\rm Kato} \in H^1(\cO_{F,S(F)}, T)$ for every $F \in \Omega(\cK)$, then we have
$$z^{\rm Kato}:=(z_F^{\rm Kato})_F \in {\rm ES}_1(T,\cK). $$
(See \cite[Lem.~6.7]{bss2}.) This Euler system is called Kato's Euler system. 

We shall describe the period-regulator isomorphism in this case. We prepare some notations. 
In the following, we assume $\#\sha(E/F)[p^\infty]<\infty$ for every $F \in \Omega(\cK)$. We abbreviate $\Upsilon(T,F)$ to $\Upsilon(F)$. For a non-negative integer $i$, we define
$$\Upsilon(F)_i^{\rm an}:= \{ \chi \in \widehat \cG_F \mid \ord_{s=1}L(E,\chi,s)=i\}$$
and 
$$\Upsilon(F)_i^{\rm alg}:= \{ \chi \in \widehat \cG_F \mid \dim_{\CC}(e_\chi(\CC\otimes_\ZZ E(F)))=i\}.$$
Then one sees by \cite[Th.~14.2(2)]{katoasterisque} that 
$$\Upsilon(F)_0^{\rm an} \subset \Upsilon(F)_0^{\rm alg} \subset \Upsilon(F)_0^{\rm alg}\cup \Upsilon(F)_{1}^{\rm alg} =\Upsilon(F).$$
(See \cite[Lem.~6.1(iii)]{bss2} for the last equality.) We define the associated idempotent by
$$e_{F,i}^\star:=\sum_{\chi \in \Upsilon(F)_i^\star} e_\chi \in \QQ[\cG_F],$$
where $\star \in \{{\rm an}, {\rm alg}\}$. Note that $e_{T,F}=e_{F,0}^{\rm alg} + e_{F,1}^{\rm alg}$. 

We first describe the $e_{F,0}^{\rm alg}$-component of the period-regulator isomorphism. It is defined by the composition
\begin{eqnarray*}
\lambda_{T,F}: e_{F,0}^{\rm alg}\left(\CC_p \otimes_{\ZZ_p}  H^1(\cO_{F,S(F)},T) \right)&\stackrel{\exp^\ast}{\simeq}& e_{F,0}^{\rm alg} \left( \CC_p \otimes_\QQ \Gamma(E, \Omega_{E/F}^1)\right)\\
&\stackrel{\alpha}{\simeq}& e_{F,0}^{\rm alg} \left(\CC_p\otimes_\QQ \left( \bigoplus_{\iota: F \hookrightarrow \CC}H_1(E^\iota(\CC), \QQ)\right)^+ \right)^\ast \\
&\stackrel{\beta}{\simeq}& e_{F,0}^{\rm alg}\left( \CC_p \otimes_{\ZZ_p} Y_F(T^\ast(1))^\ast \right),
\end{eqnarray*}
where the first isomorphism is induced by (the localization map and) the dual exponential map 
$$\exp^\ast: \bigoplus_{w\in S_p(F)} H^1_{/f}(F_w,V) \simeq \bigoplus_{w \in S_p(F)}\QQ_p\otimes_{\ZZ_p}E_1(F_w)^\ast \xrightarrow{\sim} \QQ_p \otimes_\QQ \Gamma(E,\Omega_{E/F}^1),$$
the second by the period map $\omega \mapsto (\gamma \mapsto \int_\gamma \omega)$, and the last by the comparison isomorphism $ \left(\bigoplus_{\iota: F \hookrightarrow \CC}H_1(E^\iota(\CC), \QQ_p)\right)^+ \simeq \QQ_p\otimes_{\ZZ_p}Y_F(T^\ast(1))$. 

Next, we describe the $e_{F,1}^{\rm alg}$-component. It is defined by the composition 
\begin{eqnarray*}
\lambda_{T,F}: e_{F,1}^{\rm alg}\left(\CC_p \otimes_{\ZZ_p}  H^1(\cO_{F,S(F)},T) \right)&\simeq & e_{F,1}^{\rm alg} \left(\CC_p \otimes_{\ZZ}  E(F) \right)\\
&\simeq& e_{F,1}^{\rm alg} \left(\CC_p \otimes_{\ZZ}  E(F) \right)^\ast \\
&\simeq& e_{F,1}^{\rm alg} \left(\CC_p \otimes_{\ZZ_p}\bigoplus_{w\in S_p(F)}  E_1(F_w) \right)^\ast \\
&\simeq& e_{F,1}^{\rm alg}\left( \CC_p \otimes_{\ZZ_p} Y_F(T^\ast(1))^\ast \right),
\end{eqnarray*}
where the first isomorphism is induced by the Kummer map $E(F) \to H^1(\cO_{F,S(F)},T_p(E))$ (see \cite[(21)]{bss2}), the second by the N\'eron-Tate height pairing, the third by the localization map $E(F)\to E(F_w)$, and the last by $\beta\circ \alpha \circ \exp^\ast$. 

We now relate Kato's zeta element $z_F^{\rm Kato}$ with the special element $\eta_F=\eta_{F/\QQ,S(F),\emptyset}(T)$ in Definition \ref{special}. By Kato's deep result \cite[Th.~6.6 and 9.7]{katoasterisque}, we have
$$\lambda_{T,F}(e_{F,0}^{\rm an} \cdot z_F^{\rm Kato})=e_{F,0}^{\rm an}\cdot \theta_{F/\QQ,S(F)}(E,1)\cdot b_F,$$
where $\theta_{F/\QQ,S(F)}(E,s):=\theta_{F/\QQ,S(F),\emptyset}(M^\ast(1),s-1)=\sum_{\chi \in \widehat \cG_F} L_{S(F)}(E,\chi^{-1},s)e_\chi$. So by the definition of $\eta_F$ we have
$$e_{F,0}^{\rm an}\cdot z_F^{\rm Kato}=e_{F,0}^{\rm an}\cdot \eta_F.$$
It is natural to expect
$$z_F^{\rm Kato}=\eta_F.$$
(This is the conjecture made in \cite[Conj.~6.2]{bss2}.) We remark that the equality
$$z_\QQ^{\rm Kato}=\eta_\QQ$$
is equivalent to Perrin-Riou's conjecture \cite{PR} (see \cite[Prop.~6.5]{bss2} or \cite[Prop.~2.10]{bks4}).

\subsection{The period-regulator isomorphism}\label{def per}
In this subsection, we give a general definition of the period-regulator isomorphism
$$\lambda_{T,F}: e_{T,F}\left(\CC_p \otimes_{\ZZ_p} {\bigwedge}_{\cA[\cG_F]}^r H^1(\cO_{F,S(F)},T) \right)\xrightarrow{\sim} e_{T,F}\left( \CC_p \otimes_{\ZZ_p} {\bigwedge}_{\cA[\cG_F]}^r Y_F(T^\ast(1))^\ast \right).$$

Let $\rgamma_{c,\Sigma}(\cO_{F,S(F)},T^\ast(1))$ be the $\Sigma$-modified compactly supported cohomology complex defined in \cite[\S 2.3.2]{sbA} and set
$$C_{F,S(F),\Sigma}(T):=\rhom_{\ZZ_p}(\rgamma_{c,\Sigma}(\cO_{F,S(F)}, T^\ast(1)), \ZZ_p[-2]).$$
It is well-known that $C_{F,S(F),\Sigma}(T)$ is a perfect complex of $\cA[\cG_F]$-modules,  acyclic outside degrees zero and one (under Hypothesis \ref{hypint}(i)), and the Euler characteristic is zero. By \cite[Prop.~2.22]{sbA}, we have a canonical isomorphism
$$H^0(C_{F,S(F),\Sigma}(T)) \simeq H^1_\Sigma(\cO_{F,S(F)},T)$$
and a canonical exact sequence
$$0\to H^2_\Sigma(\cO_{F,S(F)},T) \to H^1(C_{F,S(F),\Sigma}(T)) \to Y_F(T^\ast(1))^\ast \to 0.$$

Let
$$\vartheta_{T,F}: \CC_p\otimes_{\ZZ_p} {\det}_{\cA[\cG_F]}(C_{F,S(F),\Sigma}(T)) \xrightarrow{\sim} \CC_p\otimes_{\ZZ_p}\cA[\cG_F]$$
be the isomorphism used in the formulation of the equivariant Tamagawa number conjecture (see \cite[\S 3.4]{BFetnc}). We normalize $\vartheta_{T,F}$ so that the equivariant Tamagawa number conjecture for $(M^\ast(1)\otimes_K F, \cA[\cG_F])$ is equivalent to the equality
$$\vartheta_{T,F}({\det}_{\cA[\cG_F]}(C_{F,S(F),\Sigma}(T))) = \cA[\cG_F] \cdot \theta^\ast_{F/K,S(F),\Sigma}(M^\ast(1),0).$$

The map $\vartheta_{T,F}$ induces an isomorphism
$$\CC_p\otimes_{\QQ_p} {\det}_{A[\cG_F]}(\QQ_p\otimes_{\ZZ_p}H^0(C_{F,S(F),\Sigma}(T)))\xrightarrow{\sim} \CC_p\otimes_{\QQ_p} {\det}_{A[\cG_F]}(\QQ_p\otimes_{\ZZ_p}H^1(C_{F,S(F),\Sigma}(T))),$$
which becomes
\begin{multline}\label{det version}
\CC_p\otimes_{\QQ_p} {\det}_{A[\cG_F]}(H^1(\cO_{F,S(F)},V)) \xrightarrow{\sim} \CC_p \otimes_{\QQ_p} \left( {\det}_{A[\cG_F]}(H^2(\cO_{F,S(F)},V)) \otimes_{\cA[\cG_F]} {\bigwedge}_{\cA[\cG_F]}^r Y_F(T^\ast(1))^\ast \right).
\end{multline}
We define $\lambda_{T,F}$ to be the isomorphism induced by this isomorphism (note that the idempotent $e_{T,F}$ kills $H^2(\cO_{F,S(F)},V)$). 

\subsection{Extended special elements}\label{sec ext}

Let $\eta_F=\eta_{F/K,S(F),\Sigma}(T)$ be the special element for $F \in \Omega(\cK)$ in Definition \ref{special}. For later use, we study connections between this element for $F=K$ and the Tamagawa number conjecture for $M^\ast(1)$. 

First, note that by definition $\eta_K=\eta_{K/K,S,\Sigma}(T) \in \CC_p\otimes_{\ZZ_p}{\bigwedge}_{\cA}^r H^1(\cO_{K,S},T)$ can be zero. In fact, we have
$$\eta_K=0 \Leftrightarrow H^2(\cO_{K,S},V)\neq 0.$$
(This phenomenon can be regarded as a ``trivial zero (or exceptional zero) phenomenon" for Euler systems. See Remark \ref{rem trivial zero} below.) For this reason, we extend the definition of $\eta_K$ so that it always becomes non-zero. 

We set
$$e:=\dim_A(H^2(\cO_{K,S},V)).$$
We define the ``extended period-regulator isomorphism"
$$\widetilde \lambda_{T,K}: \CC_p\otimes_{\ZZ_p} {\bigwedge}_\cA^{r+e} H^1(\cO_{K,S},T) \xrightarrow{\sim} \CC_p\otimes_{\ZZ_p} \left({\bigwedge}_\cA^{e} H^2_\Sigma(\cO_{K,S},T)_{\rm tf} \otimes_\cA {\bigwedge}_\cA^r Y_K(T^\ast(1))^\ast  \right)$$
to be the isomorphism induced by \eqref{det version} (with $F=K$).

\begin{definition}\label{def ext}
Assume $H^0(K,T)=0$. Fix $\cA$-bases $x \in {\bigwedge}_\cA^e H_\Sigma^2(\cO_{K,S},T)_{\rm tf}$ and $b \in {\bigwedge}_\cA^r Y_K(T^\ast(1))^\ast$. 
We define the {\it extended special element} for $T$ by
$$\widetilde \eta_K=\widetilde \eta_{K,S,\Sigma}(T):=\widetilde \lambda_{T,K}^{-1} (L_{S,\Sigma}^\ast(M^\ast(1),0)\cdot (x\otimes b)) \in \CC_p\otimes_{\ZZ_p} {\bigwedge}_\cA^{r+e} H^1(\cO_{K,S},T).$$
\end{definition}

\begin{remark}
By definition, $\widetilde \eta_K$ is always non-zero. 
When $e=0$ (i.e., $H^2(\cO_{K,S},V)=0$), we have
$$ \eta_K=\widetilde \eta_K \neq 0.$$
\end{remark}

\begin{remark}
We can expect $H^2(\cO_{K,S},V)=0$ in many cases. For example, it is well-known that $H^2(\cO_{K,S},\QQ_p)=0$ is equivalent to the Leopoldt conjecture for $K$. Soul\'e proved $H^2(\cO_{K,S}, \QQ_p(j))=0$ for $j>1$, and Schneider conjectured that $H^2(\cO_{K,S},\QQ_p(j))=0$ for $j \leq 0$ (see \cite[p.~641]{NSW}). More generally, for a smooth projective scheme $X/K$, Jannsen conjectured that $H^2(\cO_{K,S}, H^i_{\text{\'et}}(X\times_K \overline \QQ, \QQ_p)(j))=0$ if $i+1 < j$ or $i+1>2j$ (see \cite[Conj.~1]{jannsen}). However, $H^2(\cO_{K,S},V)$ can be non-zero in the cases considered in \S \ref{sec ex} (see also Remarks \ref{rem trivial zero}, \ref{rem cnf} and \ref{rem ell} below). 
\end{remark}

\begin{remark}\label{rem trivial zero}
Let $\chi$ be a non-trivial character of $G_K$ of finite order and set $F:=\overline \QQ^{\ker \chi}$. If $T=\ZZ_p[\im \chi](1)\otimes \chi^{-1}$ and $S=S_\infty(K)\cup S_p(K) \cup S_{\rm ram}(F/K)$, then we have
$$H^2(\cO_{K,S},V) \simeq e_\chi(\QQ_p(\im \chi)\otimes_\ZZ X_{F,S_p(K)}),$$
where $X_{F,S_p(K)}$ is defined in \S \ref{rs ex}. So in this case
$$H^2(\cO_{K,S},V)= 0 \Leftrightarrow \chi({\rm Fr}_v)\neq 1 \text{ for all }v\in S_p(K).$$
This is the usual ``no trivial zeros" condition (see \cite[p.~1555]{bks2}). However, when $T=T_p(E)$ with an elliptic curve $E$, the condition $H^2(\cO_{K,S},V)\neq 0$ has nothing to do with the ``exceptional zero" phenomenon in the sense of Mazur-Tate-Teitelbaum \cite{MTT}. 
\end{remark}

\begin{remark}\label{rem cnf}
Consider the $\GG_m$ case (i.e., $M=h^0(K)(1)$ and $T=\ZZ_p(1)$). In this case, the extended special element is explicitly described as follows. First, note that there is a canonical exact sequence
$$0 \to \ZZ_p\otimes_\ZZ {\rm Cl}_S^\Sigma(K) \to H^2_\Sigma(\cO_{K,S},\ZZ_p(1)) \to \ZZ_p\otimes_\ZZ X_{K,S\setminus S_\infty(K)} \to 0,$$
where ${\rm Cl}_S^\Sigma(K)$ is the $(S,\Sigma)$-class group of $K$ (see \cite[\S 1.7]{bks1} for example) and $X_{K,S\setminus S_\infty(K)}$ is defined in \S \ref{rs ex}. So we have
$$H^2_{\Sigma}(\cO_{K,S},\ZZ_p(1))_{\rm tf}\simeq \ZZ_p\otimes_\ZZ X_{K,S\setminus S_\infty(K)}$$
and 
$$e=\# S - \# S_\infty(K)-1.$$
We set $r:=\# S_\infty(K)$ (so that $\# S=r+e+1$). We fix a labeling $S_\infty(K)=\{v_1,\ldots,v_r\}$ and $S=\{v_0,v_1,\ldots,v_{r+e}\}$. Recall that $Y_K(T^\ast(1))^\ast$ is identified with $\ZZ_p\otimes_\ZZ Y_{K,S_\infty(K)}$ (see \S \ref{rs ex}). We define $\ZZ$-bases $x \in {\bigwedge}_{\ZZ}^e X_{K,S\setminus S_\infty(K)}$ and $b \in {\bigwedge}_{\ZZ}^r Y_{K,S_\infty(K)}$ by setting
$$x:=(v_{r+1}-v_0)\wedge \cdots \wedge (v_{r+e}-v_0) \text{ and }b:=v_1\wedge \cdots \wedge v_r.$$
We assume $\Sigma$ is chosen so that the $(S,\Sigma)$-unit group $\cO_{K,S,\Sigma}^\times:=\ker (\cO_{K,S}^\times \to \bigoplus_{v\in \Sigma}(\cO_K/v)^\times)$ is torsion-free, and let $\{u_1,\ldots,u_{r+e}\}$ be a $\ZZ$-basis of $\cO_{K,S,\Sigma}^\times$. Then one sees that the extended special element is described explicitly as
$$\widetilde \eta_{K,S,\Sigma}(\ZZ_p(1))= \pm \# {\rm Cl}_S^\Sigma(K) \cdot u_1\wedge\cdots \wedge u_{r+e} \in \ZZ_p\otimes_\ZZ{\bigwedge}_\ZZ^{r+e}\cO_{K,S,\Sigma}^\times \simeq {\bigwedge}_{\ZZ_p}^{r+e}H^1_\Sigma(\cO_{K,S},\ZZ_p(1)).$$
(This also coincides with the Rubin-Stark element for $(K/K,S,\Sigma,S\setminus \{v_0\})$.) In fact, the extended period-regulator isomorphism in this case is induced by the Dirichlet regulator
$$\delta_K: \RR \otimes_\ZZ \cO_{K,S}^\times \xrightarrow{\sim} \RR \otimes_{\ZZ} X_{K,S}; \ u \mapsto -\sum_{v \in S}\log|u|_v\cdot v,$$
and the above description follows from the well-known class number formula
$$L_{S,\Sigma}^\ast(M^\ast(1),0)=\zeta_{K,S,\Sigma}^\ast(0)=\pm \# {\rm Cl}_S^\Sigma(K)\cdot \det(\log|u_i|_{v_j})_{1\leq i,j\leq r+e},$$
where $\zeta_{K,S,\Sigma}(s)$ is the usual $(S,\Sigma)$-modified Dedekind zeta function for $K$. 
\end{remark}

\begin{remark}\label{rem ell}
Consider the case when $K=\QQ$ and $M=h^1(E)(1)$ with an elliptic curve $E$ over $\QQ$. If we assume $\sha(E/\QQ)[p^\infty]<\infty$, then we have
$$e=\max\{0, {\rm rank}(E(\QQ))-1\}.$$
(See \cite[Lem.~6.1]{bss2}.) If $L(E,1)\neq 0$, then by the argument in \S \ref{kato ex} we have $e=0$ and 
$$\widetilde \eta_\QQ=\eta_\QQ =z_\QQ^{\rm Kato}.$$
(We take $\Sigma$ to be empty.) If ${\rm rank}(E(\QQ))>0$, then $\widetilde \eta_\QQ$ coincides with the ``Birch-Swinnerton-Dyer element" $\eta_x^{\rm BSD}$ defined in \cite[Def.~2.4]{bks4} (by letting $b \in Y_\QQ(T^\ast(1))^\ast$ be $e^+\delta(\xi)^\ast$ in loc. cit.). 
\end{remark}

The following is a generalization of \cite[Prop.~2.6]{bks4}.

\begin{proposition}\label{tnc ext}
Assume that $H^0(K,T)=0$ and that either $\Sigma$ is non-empty or $H^1(\cO_{K,S},T)$ is $\cA$-free. 
Then the Tamagawa number conjecture for $M^\ast(1)$ (with coefficients in $\cA$) holds if and only if we have an equality of $\cA$-modules
$$\cA \cdot \widetilde \eta_K = {\rm Fitt}_{\cA}(H_\Sigma^2(\cO_{K,S},T)_{\rm tors})\cdot {\bigwedge}_\cA^{r+e}H^1_\Sigma(\cO_{K,S},T). $$
In particular, the Tamagawa number conjecture for $M^\ast(1)$ implies the ``integrality" of $\widetilde \eta_K$:
$$\widetilde \eta_K \in {\bigwedge}_\cA^{r+e}H^1_\Sigma(\cO_{K,S},T). $$
\end{proposition}

\begin{proof}
Consider the following map:
\begin{eqnarray*}
{\det}_\cA(C_{K,S,\Sigma}(T)) &\hookrightarrow& \QQ_p \otimes_{\ZZ_p} {\det}_\cA(C_{K,S,\Sigma}(T)) \\
&\simeq& {\det}_A(H^1_{\Sigma}(\cO_{K,S},V)) \otimes_A {\det}_A^{-1}(H^2_{\Sigma}(\cO_{K,S},V)) \otimes_\cA {\det}_\cA^{-1}(Y_K(T^\ast(1))^\ast)\\
&=& \QQ_p\otimes_{\ZZ_p} \left( {\bigwedge}_\cA^{r+e} H_\Sigma^1(\cO_{K,S},T)  \otimes_\cA {\bigwedge}_\cA^e H_\Sigma^2(\cO_{K,S},T)^\ast \otimes_\cA {\bigwedge}_\cA^r Y_K(T^\ast(1))\right)\\
&\simeq& \QQ_p\otimes_{\ZZ_p}  {\bigwedge}_\cA^{r+e} H^1_\Sigma(\cO_{K,S},T),
\end{eqnarray*}
where the last isomorphism is defined by using the fixed $\cA$-bases of ${\bigwedge}_\cA^e H^2_\Sigma(\cO_{K,S},T)_{\rm tf}$ and ${\bigwedge}_\cA^r Y_K(T^\ast(1))^\ast$. The proposition follows by noting that the image of this map is 
$${\rm Fitt}_{\cA}(H_\Sigma^2(\cO_{K,S},T)_{\rm tors})\cdot {\bigwedge}_\cA^{r+e}H^1_\Sigma(\cO_{K,S},T).$$
\end{proof}

\section{The Iwasawa main conjecture for motives}

In this section, we study Iwasawa theory for motives. In \S \ref{formulate imc}, we give a formulation of the (equivariant) Iwasawa main conjecture (see Conjecture \ref{IMC}). In \S \ref{section nonequiv}, we give another formulation in the ``non-equivariant" case (see Conjecture \ref{neIMC}). 
In \S \ref{subsec:statement}, we state a theorem that gives us an approach to prove ``one half" of the Iwasawa main conjecture under standard hypotheses (see Theorem \ref{main}). 
\S \ref{subsec:outline} is devoted to the proof of Theorem \ref{main}.

Throughout this section, we let $K$ be a number field and $p > 2$ an odd prime number.
Let $\cA$ be the ring of integers of a finite extension of $\QQ_p$. 
Let $T$ be a free $\cA$-module of finite rank equipped with a continuous action of the absolute Galois group $G_K$ of $K$ which is unramified outside a finite set of places of $K$.

\subsection{Formulation of the Iwasawa main conjecture}\label{formulate imc}
We fix the following data:
\begin{itemize}
\item $L/K$: a finite abelian extension in which all $v\in S_\infty(K)$ split completely;
\item $K_\infty/K$: a $\Z_p^d$-extension with $d \geq 1$;
\item $S$: a finite set of places of $K$ containing $S_\infty(K)\cup S_p(K)\cup S_{\rm ram}(L/K) \cup S_{\rm ram}(T)$;
\item $\Sigma$: a finite set of places of $K$ such that $S \cap \Sigma=\emptyset$. 
\end{itemize}

We set some notations attached to these data. We set
$$L_\infty:=L \cdot K_\infty, \ \cG_\infty:=\Gal(L_\infty/K) \text{ and }\Lambda=\Lambda_{L_\infty}:=\cA[[\cG_\infty]].$$
We set
$$\TT:= T \otimes_{\cA}\Lambda,$$
on which $G_K$ acts by
$$\sigma\cdot (t \otimes \lambda):=\sigma t \otimes \lambda \overline \sigma^{-1} \quad (\sigma\in G_K, \ t \in T, \ \lambda \in \Lambda),$$
where $\overline \sigma^{-1} \in \cG_\infty$ denotes the image of $\sigma^{-1}\in G_K$ under the natural surjection $G_K \twoheadrightarrow \cG_\infty$. 

We keep assuming Hypothesis \ref{hypint}, and let $r=r_T:={\rm rank}_\cA(Y_K(T^\ast(1)))$ be the basic rank (see Definition \ref{def basic}). 
It is well-known that $\rgamma_{\Sigma}(\cO_{K,S},\TT)$ is a perfect complex of $\Lambda$-modules, which is acyclic outside degrees one and two, and that there is a non-canonical  isomorphism
\begin{equation}\label{noncan}
Q(\Lambda)\otimes_\Lambda H^1_\Sigma(\cO_{K,S},\TT) \simeq Q(\Lambda) \otimes_{\Lambda} (H^2_\Sigma(\cO_{K,S},\TT)\oplus \Lambda^r).
\end{equation}
Here $Q(\Lambda)$ denotes the total quotient ring of $\Lambda$. 
In other words, the Euler characteristic of the complex $\rgamma_{\Sigma}(\cO_{K,S},\TT)$ is $-r$.
We now assume the following. 

\begin{hypothesis}[The weak Leopoldt conjecture] \label{leop}
$H^2_\Sigma(\cO_{K,S},\TT)$ is $\Lambda$-torsion, i.e., 
$$Q(\Lambda)\otimes_{\Lambda} H_\Sigma^2(\cO_{K,S},\TT)=0.$$
\end{hypothesis}

\begin{remark}
If $K_\infty/K$ is the cyclotomic $\ZZ_p$-extension, then it is expected that Hypothesis \ref{leop} is always satisfied (see \cite[\S 1.3]{PRast}). When $T=\ZZ_p(1)$, by a well-known theorem of Iwasawa, Hypothesis \ref{leop} is satisfied if no finite place of $K$ splits completely in $K_\infty$ (in particular, it is satisfied if $K_\infty/K$ is the cyclotomic $\ZZ_p$-extension). When $K=\QQ$ and $T=T_p(E)$ with an elliptic curve $E$ over $\QQ$, Hypothesis \ref{leop} is proved by Kato \cite[Th.~12.4(1)]{katoasterisque}. For the case when $K$ is imaginary quadratic and $K_\infty/K$ is the anticyclotomic $\ZZ_p$-extension, see Remark \ref{anti leop}. 
\end{remark}

Under Hypothesis \ref{leop}, we see by \eqref{noncan} that 
$$Q(\Lambda)\otimes_\Lambda H^1_\Sigma(\cO_{K,S},\TT) \simeq Q(\Lambda)^r.$$
In particular, we have a canonical isomorphism
\begin{equation}\label{canisom}
Q(\Lambda)\otimes_\Lambda {\det}_\Lambda^{-1}(\rgamma_\Sigma(\cO_{K,S},\TT)) \simeq Q(\Lambda)\otimes_\Lambda {\bigwedge}_\Lambda^r H^1_\Sigma(\cO_{K,S},\TT).
\end{equation}

Since we have $Q(\Lambda)\otimes_\Lambda {\bigcap}_\Lambda^r H\simeq Q(\Lambda) \otimes_\Lambda {\bigwedge}_\Lambda^r H$ for any finitely generated $\Lambda$-module $H$, we obtain a canonical isomorphism
\begin{equation}\label{canisom2}
Q(\Lambda)\otimes_\Lambda {\det}_\Lambda^{-1}(\rgamma_\Sigma(\cO_{K,S},\TT)) \simeq Q(\Lambda)\otimes_\Lambda {\bigcap}_\Lambda^r H^1_\Sigma(\cO_{K,S},\TT).
\end{equation}

We need the following lemma proved by Sakamoto in \cite[Lem.~B.15]{sakamoto}, which is used frequently in this paper. 

\begin{lemma}\label{lemlimit}
Assume Hypothesis \ref{hypint}. Then there
is a canonical isomorphism
$${\bigcap}_\Lambda^r H^1_\Sigma(\cO_{K,S},\TT)\simeq \varprojlim_{F \in \Omega(L_\infty)} {\bigcap}_{\cA[\cG_{F}]}^r H^1_\Sigma(\cO_{F,S},T).$$
(Recall that $\Omega(L_\infty)$ denotes the set of finite subextensions $F/K$ of $L_\infty/K$ and $\cG_F:=\Gal(F/K)$.) 
\end{lemma}


We shall now formulate the Iwasawa main conjecture. We assume that a certain  canonical element
$$c_{L_{\infty}} \in {\bigcap}_\Lambda^r H^1_\Sigma(\cO_{K,S},\TT)$$
is given. By Lemma \ref{lemlimit}, this is equivalent to assuming that a canonical Euler system
$$c \in {\rm ES}_r(T,L_\infty)=\varprojlim_{F \in \Omega(L_\infty)} {\bigcap}_{\cA[\cG_F]}^r H^1_\Sigma(\cO_{F,S},T)$$
is given. A reasonable way is to assume that $T$ comes from a motive $M$ as in \S \ref{sec euler} and assume Conjecture \ref{conjes}, but we do not need to assume it. For example, in the elliptic curve case, one can take $c$ to be Kato's Euler system, although we do not know if it satisfies the properties (i) and (ii) in Conjecture \ref{conjes} (see Remark \ref{rem strong} and \S \ref{kato ex}).

\begin{conjecture}[The Iwasawa main conjecture] \label{IMC}
Assume Hypotheses \ref{hypint} and \ref{leop}. Then there exists a (unique) $\Lambda$-basis
$$\fz_{L_{\infty}} \in {\det}_\Lambda^{-1}(\rgamma_\Sigma(\cO_{K,S},\TT))$$
such that the isomorphism
$$Q(\Lambda)\otimes_\Lambda {\det}_\Lambda^{-1}(\rgamma_\Sigma(\cO_{K,S},\TT)) \stackrel{\eqref{canisom2}}{\simeq} Q(\Lambda)\otimes_\Lambda {\bigcap}_\Lambda^r H^1_\Sigma(\cO_{K,S},\TT)$$
sends $\fz_{L_{\infty}}$ to $c_{L_{\infty}}$. 
\end{conjecture}

\begin{remark}
In the $\GG_m$ case (i.e., $M=h^0(K)(1)$ and $c = \eta^{\rm RS}$: see \S \ref{rs ex}), Conjecture \ref{IMC} is equivalent to ${\rm IMC}(L_\infty/K,S,\Sigma)$ in \cite[Conj.~3.1]{bks2} (see also \cite[Rem.~3.6]{bks2}). 
\end{remark}

\begin{remark}
If $K=\QQ$, $M=h^1(E)(1)$ with an elliptic curve $E$ over $\QQ$ and $c=z^{\rm Kato}$ (see \S \ref{kato ex}), then Conjecture \ref{IMC} in the case $L=K$ is equivalent to Kato's main conjecture \cite[Conj.~12.10]{katoasterisque} (see \cite[Rem.~7.2]{bks4}). In the ``equivariant" case (i.e., $L$ is a general abelian extension of $K$), Conjecture \ref{IMC} is studied by the first author in \cite{kataoka1} and \cite{kataoka2}. 
\end{remark}

\begin{remark}\label{rem imc1}
Using the language introduced in \cite{sbA}, we can rephrase Conjecture \ref{IMC} as follows. The module of ``vertical determinantal systems" introduced in loc. cit. for $L_\infty$ is simply defined by
$${\rm VS}(T,L_\infty):={\det}_\Lambda^{-1}(\rgamma_\Sigma(\cO_{K,S},\TT)).$$
Then by \cite[Th.~2.18]{sbA} there is a canonical map
$$\Theta_{T,L_\infty}: {\rm VS}(T,L_\infty) \to {\rm ES}_r(T,L_\infty)$$
which induces \eqref{canisom2}. (Namely, one can show that the image of ${\det}_\Lambda^{-1}(\rgamma_\Sigma(\cO_{K,S},\TT))$ under the map \eqref{canisom2} lies in ${\bigcap}_\Lambda^r H^1_\Sigma(\cO_{K,S},\TT)$.) Thus Conjecture \ref{IMC} predicts the existence of a $\Lambda$-basis
$$\fz_{L_{\infty}} \in {\rm VS}(T,L_\infty)$$
such that
$$\Theta_{T,L_\infty}(\fz_{L_{\infty}})=c_{L_{\infty}}.$$

We note that the map $\Theta_{T,L_\infty}$ can be defined without assuming Hypothesis \ref{leop}. However, if Hypothesis \ref{leop} is not satisfied, then one can show that $\Theta_{T,L_\infty}$ is not injective.
\end{remark}

\begin{remark}\label{remark sbA}
Let $\cK/K$ be an abelian extension as in \S \ref{sec euler}. The observation in Remark \ref{rem imc1} can be generalized naturally for $\cK$: under Hypothesis  \ref{hypint}, one can define an $\cA[[\Gal(\cK/K)]]$-module ${\rm VS}(T,\cK)$, which is free of rank one, and a canonical map
$$\Theta_{T,\cK}: {\rm VS}(T,\cK) \to {\rm ES}_r(T,\cK).$$
It is natural to expect that, for a given canonical Euler system $c \in {\rm ES}_r(T,\cK)$, there exists an $\cA[[\Gal(\cK/K)]]$-basis
$$\fz \in {\rm VS}(T,\cK)$$
such that
$$\Theta_{T,\cK}(\fz)=c.$$
This gives a generalization of Conjecture \ref{IMC}. (Namely, Conjecture \ref{IMC} is the special case of this prediction for $\cK=L_\infty$.) We note that, by the work of Burns-Greither \cite{BG}, this generalization of Conjecture \ref{IMC} is known to be true when $K=\QQ$ and $c$ is the cyclotomic unit Euler system (see \S \ref{rs ex} and \cite[Lem.~5.2]{bdss}). 
\end{remark}

\subsection{The ``non-equivariant" Iwasawa main conjecture}\label{section nonequiv}
In this subsection, we give another formulation of Conjecture \ref{IMC} in the ``non-equivariant" case, i.e., when $L=K$. 
We moreover assume that $K_{\infty}/K$ is a $\Z_p$-extension, that is, $d = 1$.
In this case, the Iwasawa algebra
$\Lambda=\Lambda_{K_\infty}:=\cA[[\Gal(K_\infty/K)]]$
is a two-dimensional regular local ring.

As before, we assume that a canonical element (Euler system)
$$c_{K_{\infty}} \in {\bigcap}_\Lambda^r H^1_\Sigma(\cO_{K,S},\TT) \simeq \varprojlim_n {\bigcap}_{\cA[\cG_{K_n}]}^r H^1_\Sigma(\cO_{K_n,S},T)= {\rm ES}_r(T,K_\infty)$$
is given, where $K_n$ denotes the $n$-th layer of the $\ZZ_p$-extension $K_\infty/K$. 

We propose the following. 

\begin{conjecture}[The non-equivariant Iwasawa main conjecture]\label{neIMC}
Assume Hypotheses \ref{hypint} and \ref{leop}. Then we have
$${\rm char}_\Lambda\left({\bigcap}_\Lambda^r H^1_\Sigma(\cO_{K,S},\TT)/ \Lambda\cdot c_{K_{\infty}} \right)
={\rm char}_\Lambda(H^2_\Sigma(\cO_{K,S},\TT)). $$
\end{conjecture}

\begin{proposition}\label{imc equivalent}
Conjecture \ref{IMC} for $L=K$ and $d=1$ is equivalent to Conjecture \ref{neIMC}. 
\end{proposition}

To prove this proposition, we need the following algebraic lemma. 

\begin{lemma}\label{alg lemma}
Let $\Lambda$ be a ring isomorphic to the formal power series ring $\cA[[X]]$. Let $Q(\Lambda)$ be the quotient field of $\Lambda$. Let $H$ be a finitely generated $\Lambda$-module and set $r:={\rm dim}_{Q(\Lambda)}(Q(\Lambda)\otimes_\Lambda H)$. 
\begin{itemize}
\item[(i)] The $\Lambda$-module ${\bigcap}_\Lambda^r H$ is free of rank one and the natural map
$${\bigcap}_\Lambda^r H \to Q(\Lambda)\otimes_\Lambda {\bigcap}_\Lambda^r H\simeq Q(\Lambda)\otimes_\Lambda {\bigwedge}_\Lambda^r H$$
is injective. 
\item[(ii)] The image of the canonical map
$${\det}_\Lambda(H) \to Q(\Lambda)\otimes_{\Lambda}{\det}_\Lambda(H) \simeq Q(\Lambda)\otimes_{\Lambda}{\bigwedge}_\Lambda^r H$$
coincides with
$${\rm char}_\Lambda(H_{\rm tors})^{-1}\cdot {\bigcap}_\Lambda^r H,$$
where $H_{\rm tors}$ is the $\Lambda$-torsion submodule of $H$, and we regard ${\bigcap}_\Lambda^r H \subset Q(\Lambda)\otimes_\Lambda {\bigwedge}_\Lambda^r H$ by (i). 
\end{itemize}
\end{lemma}
\begin{proof}
Since we have $Q(\Lambda)\otimes_\Lambda {\bigcap}_\Lambda^r H \simeq Q(\Lambda)\otimes_\Lambda {\bigwedge}_\Lambda^r H$, we see that the $\Lambda$-rank of ${\bigcap}_\Lambda^r H$ is one. The $\Lambda$-module ${\bigcap}_\Lambda^r H$ is actually free, since it is reflexive by definition (see \cite[Cor.~5.1.3 and Prop.~5.1.9]{NSW}). This proves (i). 

To prove (ii), it is sufficient to show that for any height one prime $\fp$ of $\Lambda$ we have
$$\im \left( {\det}_{\Lambda_\fp}(H_\fp)\to Q(\Lambda)\otimes_{\Lambda_\fp}{\bigwedge}_{\Lambda_\fp}^r H_\fp \right) = {\rm Fitt}_{\Lambda_\fp}(H_{\fp, {\rm tors}})^{-1}\cdot \left( {\bigcap}_\Lambda^r H\right)_\fp.$$
Since $\Hom_{\Lambda_\fp}(H_\fp,\Lambda_\fp)=\Hom_{\Lambda_\fp}(H_{\fp, {\rm tf}}, \Lambda_\fp)$, we have an identification
$$\left({\bigcap}_\Lambda^r H \right)_\fp=  {\bigwedge}_{\Lambda_\fp}^r H_{\fp, {\rm tf}} \quad \text{in}\quad Q(\Lambda)\otimes_{\Lambda_\fp}{\bigwedge}_{\Lambda_\fp}^r H_\fp.$$
The claim follows from the following well-known fact: for a discrete valuation ring $R$ and a finitely generated $R$-module $M$, we have
$$\im \left( {\det}_R(M) \rightarrow Q(R) \otimes_R {\bigwedge}_R^s M\right)= {\rm Fitt}_R(M_{\rm tors})^{-1} \cdot {\bigwedge}_R^s M_{\rm tf},$$
where $s:=\dim_{Q(R)}(Q(R)\otimes_R M)$.

\end{proof}

\begin{proof}[Proof of Proposition \ref{neIMC}]
Since $\Lambda$ is regular, we have a canonical isomorphism
$${\det}_\Lambda^{-1}(\rgamma_\Sigma(\cO_{K,S},\TT)) \simeq {\det}_\Lambda (H^1_\Sigma(\cO_{K,S},\TT)) \otimes_\Lambda {\det}_\Lambda^{-1}(H^2_\Sigma(\cO_{K,S},\TT)).$$
By Lemma \ref{alg lemma}, under \eqref{canisom2}, this module corresponds to
$${\rm char}_\Lambda(H^2_\Sigma(\cO_{K,S},\TT)) \cdot {\bigcap}_\Lambda^r H^1_\Sigma(\cO_{K,S},\TT).$$
The claim easily follows from this. 
\end{proof}

\subsection{Deduction of one half of the main conjecture}\label{subsec:statement}

In this subsection, we state Theorem \ref{main} below, which is one of the main results in this paper.
It gives a general strategy to solve Conjecture \ref{IMC}.
The proof will be given in the next subsection.

We keep the preceding notations.
For simplicity, we assume $\Sigma = \emptyset$.
Note that then, for each abelian extension $\cK/K$, Hypothesis \ref{hypint} is equivalent to $H^0(\cK, T/p) = 0$.
We will essentially consider the case where $\cK$ contains $L_{\infty}$ and $\cK/L_{\infty}$ is a pro-$p$ extension, and in that case  $H^0(\cK, T/p) = 0$ is equivalent to $H^0(L, T/p) = 0$, which will be a part of Hypothesis \ref{hyp:2} below.
%

We give a list of hypotheses.
Let $(-)^{\vee}$ denote the Pontryagin dual.

%

\begin{hypothesis}\label{hyp:nonanom}
For every $\fq \in S \setminus S_{\infty}(K)$, we have
\[
H^0(L \otimes_K K_{\fq}, T^{\vee}(1)) = 0,
\]
which is by the local duality equivalent to
\[
H^2(L \otimes_K K_{\fq}, T) = 0.
\]
\end{hypothesis}


\begin{hypothesis}[{\cite[Hyp.~3.3]{bss}}]\label{hyp:2}
We have
\[
H^0(L, T/p) = 0
\]
and 
\[
H^0(L, (T/p)^{\vee}(1)) = 0.
\]
\end{hypothesis}


For each integers $m \geq 1$ and $n \geq 0$, we put
\[
R_{m, n} = \cA/p^m[\Gal(L_n/K)],
\]
where $L_n$ is the $n$-th layer of $L_{\infty}/L$ (i.e., $\Gal(L_{\infty}/L_n) = \Gal(L_{\infty}/L)^{p^n}$ and $\Gal(L_n/L)\simeq (\ZZ/p^n)^d$).
Then $R_{m, n}$ is a zero-dimensional Gorenstein ring, and we have $\Lambda = \varprojlim_{m, n} R_{m, n}$.
We also put
\[
T_{m, n} = T \otimes_{\cA} R_{m, n},
\]
which we regard as a Galois representation over $R_{m, n}$ in the same way as $\bT = T \otimes \Lambda$.

As in \cite[\S 3.1.2]{bss}, we use the following notation.
We put
\[
K_{p^m} = K(\mu_{p^m}, (\cO_K^{\times})^{1/p^m}) K(1),
\]
where $K(1)$ denotes the maximal $p$-extension in the Hilbert class field of $K$.
For each $m, n$, let $K(T_{m, n})$ be the minimal Galois extension of $K$ such that the action of $G_K$ on $T_{m, n}$ factors through $\Gal(K(T_{m, n})/K)$.
We put $K(T_{m, n})_{p^m} = K(T_{m, n}) K_{p^m}$.

\begin{hypothesis}[{\cite[Hyp.~3.2]{bss}}]\label{hyp:1}\ 
\begin{itemize}
\item[(i)]
The residual representation of $T$ is irreducible as a representation of $G_K$.
\item[(ii)]
For every $m \geq 1$ and $n \geq 0$,
there exists $\tau \in G_{K_{p^m}}$ such that $T_{m, n} / (\tau - 1) T_{m, n}$ is a free $R_{m, n}$-module of rank one.
\item[(iii)]
For every $m \geq 1$, $n \geq 0$, we have
\[
H^1(K(T_{m, n})_{p^m}/K, T_{m, n}) = 0
\]
and
\[
H^1(K(T_{m, n})_{p^m}/K, T_{m, n}^{\vee}(1)) = 0.
\]
\end{itemize}
\end{hypothesis}

The following result gives a sufficient condition for Hypotheses \ref{hyp:2} and \ref{hyp:1} to be satisfied. 

\begin{proposition}
Put $a := \rank_{\cA}(T)$ and suppose $a \geq 2$.
Then Hypotheses \ref{hyp:2} and \ref{hyp:1} are satisfied if $p \geq 5$, $(a, p-1) \neq 1$, and the image of the Galois representation
\[
G_K \to \Aut_{\cA}(T) \simeq \GL_{a}(\cA)
\]
contains $\SL_{a}(\Z_p)$.
\end{proposition}
\begin{proof}
This result will not be used in this paper, so the reader may skip the proof.

The basic idea is the same as \cite[Lem.~6.17(ii)]{bss2}.
We first observe the following claim: for any solvable extension $M/K$ (i.e., a Galois extension whose Galois group is solvable), the image of the homomorphism
\[
G_M \to \Aut_{\cA}(T) \simeq \GL_{a}(\cA)
\]
also contains $\SL_{a}(\Z_p)$.
This follows from the fact that $\SL_a(\Z_p)$ is a perfect group since we assume $p \geq 5$.

Let $\pi$ be a uniformizer of $\cA$.
Then the above claim implies that, for any solvable extension $M/K$, the image of the homomorphism
\[
G_M \to \Aut_{\cA}(T/\pi) \simeq \GL_{a}(\cA/\pi)
\]
contains $\SL_a(\bF_p)$.
It is easy to see that the action of $\SL_a(\bF_p)$ on $(\cA/\pi)^{\oplus a}$ is irreducible.
From these facts, Hypothesis \ref{hyp:2} and Hypothesis \ref{hyp:1}(i) follow immediately.

Let us show Hypothesis \ref{hyp:1}(ii).
By the above claim, we can take an element $\tau \in G_{L_n K_{p^m}}$ such that $\tau$ acts on $T \simeq \cA^{\oplus a}$ as the matrix
\[
\begin{pmatrix}
1 & 1 & &  &  \\
 & 1 & 1 & & \\
& & \ddots & \ddots& \\
 &   & & 1 & 1 \\
 & &  & & 1\\
\end{pmatrix}.
\]
Then $T/(\tau - 1) T$ is a free $\cA$-module of rank one.
Moreover, since $\tau \in G_{L_n}$, we have 
\[
T_{m, n}/(\tau - 1)T_{m, n} \simeq T/(\tau - 1)T \otimes_{\cA} R_{m, n}.
\]
Therefore, this $\tau$ satisfies the condition of Hypothesis \ref{hyp:1}(ii).

Finally let us show Hypothesis \ref{hyp:1}(iii).
We have the inflation-restriction exact sequence
\[
H^1(K_{p^m}/K, H^0(K_{p^m}, T_{m, n}))
\to H^1(K(T_{m, n})_{p^m}/K, T_{m, n})
\to H^1(K(T_{m, n})_{p^m}/K_{p^m}, T_{m, n}).
\]
The first term here vanishes because we have
\[
H^0(K_{p^m}, T_{m, n}) 
\subset H^0(L_n K_{p^m}, T_{m, n}) 
\simeq H^0(L_n K_{p^m}, T/p^m) \otimes_{\cA/p^m} R_{m, n}
\]  
and the above irreducibility of the Galois representation implies that $H^0(L_n K_{p^m}, T/p^m) = 0$.
Moreover, we can show that the last term of the sequence also vanishes in the following manner.
The action of $\Gal(K(T_{m, n})_{p^m}/K_{p^m})$ on $T_{m, n}$ is presented by a homomorphism
\[
\Gal(K(T_{m, n})_{p^m}/K_{p^m})
\hookrightarrow \Aut_{R_{m, n}}(T_{m, n}) 
\simeq \GL_a(R_{m, n}).
\]
The first claim in this proof implies that the image of this homomorphism contains $\SL_a(\Z_p/p^m)$.
By the assumption $(a, p - 1) \neq 1$, there exists an element $\lambda \in \Z_p^{\times}$ such that $\lambda \neq 1$ and $\lambda^{a} = 1$.
Then we may use the ``center kills'' argument for the scalar matrix $\lambda \in \SL_{a}(\Z_p/p^m) \subset \GL_a(R_{m, n})$ and obtain $H^1(K(T_{m, n})_{p^m}/K_{p^m}, T_{m, n}) = 0$ as claimed. 
\end{proof}

\begin{hypothesis}[{\cite[Hyp.~6.11]{bss}}]\label{hyp:6}
$\Frob_{\fq}^{p^k} - 1$ is injective on $T$ for every finite place $\fq \not \in S$ and $k \geq 0$.
\end{hypothesis}

Recall that, for an abelian extension $\cK/K$, the module of Euler systems $\ES_{r}(T, \cK)$ is defined in Definition \ref{defn:ES} (we take $\Sigma = \emptyset$).
When $\cK \supset L_{\infty}$, we can consider the composite map
\[
\ES_{r}(T, \cK) 
\to \ES_{r}(T, L_{\infty}) 
\simeq {\bigcap}_{\Lambda}^{r} H^1(\cO_{K, S}, \bT),
\]
where the first map is the natural restriction map and the last isomorphism follows from Lemma \ref{lemlimit}.
For each $c = (c_F)_F \in \ES_{r}(T, \cK)$, we write $c_{L_{\infty}} \in {\bigcap}_\Lambda^r H^1(\cO_{K,S},\TT)$ for the image of $c$.

The following is the main result of this section, which roughly says that ``one half" of the Iwasawa main conjecture (Conjecture \ref{IMC}) holds for any Euler system $c \in \ES_{r}(T, \cK)$.

\begin{theorem}\label{main}
Let $\cK$ be an abelian extension of $K$ that contains $L_{\infty}$ and $K(\fq)$ for each finite place $\fq \not \in S$, where $K(\fq)$ denotes the maximal $p$-extension in the ray class field of $K$ modulo $\fq$.
%
Let us assume that Hypotheses \ref{leop}, \ref{hyp:nonanom}, \ref{hyp:2}, \ref{hyp:1}, and \ref{hyp:6} hold.
We assume $r\geq 1$ and $p \geq 5$. 
Then, for any Euler system
\[
c \in \ES_{r}(T, \cK),
\]
there exists an element
$$\fz_{L_{\infty}} \in {\det}_\Lambda^{-1}(\rgamma(\cO_{K,S},\TT))$$
such that the isomorphism
$$Q(\Lambda)\otimes_\Lambda {\det}_\Lambda^{-1}(\rgamma(\cO_{K,S},\TT)) \stackrel{\eqref{canisom2}}{\simeq} Q(\Lambda)\otimes_\Lambda {\bigcap}_\Lambda^r H^1(\cO_{K,S},\TT)$$
sends $\fz_{L_{\infty}}$ to $c_{L_{\infty}}$. 
\end{theorem}

%

\subsection{Proof of Theorem \ref{main}}\label{subsec:outline}

This subsection is devoted to the proof of Theorem \ref{main}.
One of the main ingredients is the work \cite{bss} by Burns, Sakamoto, and the second-named author.
Another is the work \cite{kataoka2} by the first-named author.

\subsubsection{Reformulation using basic elements}\label{subsubsec:basic}

The first step is to reformulate Theorem \ref{main} in terms of basic elements that are introduced in \cite{kataoka2}.
It will be clear that the reformulation has a similar background to that in Remark \ref{rem imc1}.

Recall that the complex $\RG(\cO_{K, S}, \bT)$ is perfect over $\Lambda$.
%
By the Euler-Poincare characteristic formula, the Euler characteristic of $\RG(\cO_{K, S}, \bT))$ equals to the basic rank $r = r_T$.

\begin{definition}[{\cite[Def.~3.1 and 3.2]{kataoka2}}]\label{defn:basic}
Suppose that $H^0(L, T/p) = 0$ holds.
This implies that, for each $m \geq 1, n \geq 0$, the complex $\RG(\cO_{K, S}, T_{m, n})$ is acyclic outside degrees one and two.
Then we have a natural homomorphism
\[
\Pi_{m, n}: {\det}_{R_{m, n}}^{-1}(\RG(\cO_{K, S}, T_{m, n})) 
	\to {\bigcap}_{R_{m, n}}^{r} H^1(\cO_{K, S}, T_{m, n})
\]
for each $m, n$ (see \cite[Def.~3.1]{kataoka2}).
Then an element of ${\bigcap}_{R_{m, n}}^{r} H^1(\cO_{K, S}, T_{m, n})$ is called {\it basic} (resp.~{\it primitive basic}) if it is the image of an element (resp.~a basis) of ${\det}_{R_{m, n}}^{-1}(\RG(\cO_{K, S}, T_{m, n}))$ under $\Pi_{m, n}$.

Note that, by Sakamoto \cite[Lem.~B.15]{sakamoto} as in Lemma \ref{lemlimit}, we have a natural isomorphism
\begin{equation}\label{eq:Hlim}
{\bigcap}_{\Lambda}^{r} H^1(\cO_{K, S}, \bT)
\simeq \varprojlim_{m, n} {\bigcap}_{R_{m, n}}^{r} H^1(\cO_{K, S}, T_{m, n}).
\end{equation}
Then, taking the limit of $\Pi_{m, n}$, we obtain a natural homomorphism
\[
\Pi: {\det}_{\Lambda}^{-1}(\RG(\cO_{K, S}, \bT)) \to {\bigcap}_{\Lambda}^{r} H^1(\cO_{K, S}, \bT).
\]
In a similar way as above, we define the notion of (primitive) basic elements in ${\bigcap}_{\Lambda}^{r} H^1(\cO_{K, S}, \bT)$. 
\end{definition}

In fact, the homomorphism $\Pi$ can be identified with $\Theta_{T, L_{\infty}}$ in Remark \ref{rem imc1}.
Therefore, we can reformulate Conjecture \ref{IMC} as follows, respecting the spirit of \cite{kataoka2}.

\begin{conjecture}\label{conj:EIMC}
The given canonical element $c_{L_{\infty}} \in {\bigcap}_{\Lambda}^{r} H^1(\cO_{K, S}, \bT)$ as in \S \ref{formulate imc} is primitive basic for $\RG(\cO_{K, S}, \bT)$.
\end{conjecture}

If we assume Hypothesis \ref{leop} (the weak Leopoldt conjecture), then Conjectures \ref{IMC} and \ref{conj:EIMC} are equivalent.
Moreover, as in Remark \ref{rem imc1}, the homomorphism $\Pi$ is injective if and only if Hypothesis \ref{leop} holds.
In the rest of this section, we do not assume Hypothesis \ref{leop}.
In fact, a general strategy to prove it is by deducing from Theorem \ref{thm:goal} below and that the annihilator ideal of $c_{L_{\infty}}$ vanishes.

%

In the same way as the equivalence between Conjectures \ref{IMC} and \ref{conj:EIMC}, we can now restate Theorem \ref{main} as follows.

\begin{theorem}\label{thm:goal}
Assume the same hypotheses as in Theorem \ref{main} except for Hypothesis \ref{leop}.
Then every element in the image of the natural map
\[
\ES_{r}(T, \cK) \to {\bigcap}_{\Lambda}^{r} H^1(\cO_{K, S}, \bT)
\]
that sends $c$ to $c_{L_{\infty}}$ is basic.
\end{theorem}

The rest of this subsection is devoted to the proof of Theorem \ref{thm:goal}.
First let us observe that we may assume that $\cK/L$ is a pro-$p$ extension.
For, writing $\cK'$ for the maximal pro-$p$ subextension of $\cK/L$, the extension $\cK'/K$ also satisfies the conditions in Theorem \ref{main}, and that the homomorphism $\ES_r(T, \cK) \to {\bigcap}_{\Lambda}^{r} H^1(\cO_{K, S}, \bT)$ factors through $\ES_r(T, \cK')$.

Let us outline the proof of Theorem \ref{thm:goal}.
We assume that $\cK/L$ is a pro-$p$ extension.
By \cite[Prop.~3.3]{kataoka2}, the theorem is equivalent to that, for each $m \geq 1$ and $n \geq 0$, every element in the image of 
\[
\ES_{r}(T, \cK) \to \ES_{r}(T, L_{\infty}) \to {\bigcap}_{R_{m, n}}^{r} H^1(\cO_{K, S}, T_{m, n})
\]
is basic.
We shall prove this by constructing the following natural commutative diagram whose upper diagonal arrow is the last displayed map:
\begin{equation}\label{eq:diagram}
\xymatrix{
	\ES_{r}(T, \cK) \ar[rd] \ar[d]_{\cD_r}
	& \\
	\KS_{r}(T_{m, n}, \PP(T_{m, n})) \ar[r]
	& {\bigcap}_{R_{m, n}}^{r} H^1(\cO_{K, S}, T_{m, n})\\
	\StS_{r}(T_{m, n}, \PP(T_{m, n})) \ar[u]_-{\simeq}^{\Reg_r} \ar[r]_-{\simeq} \ar[ru]
	& {\det}_{R_{m, n}}^{-1}(\RG(\cO_{K, S}, T_{m, n})). \ar[u]_{\Pi_{m, n}}
}
\end{equation}
Here $\KS_r$ denotes the module of Kolyvagin systems, $\StS_r$ the module of Stark systems, $\cD_r$ the Kolyvagin derivative map, $\Reg_r$ the regulator map, and the other unnamed arrows are all natural maps (see the subsequent discussion for more detailed definitions).
It is clear that this diagram proves the theorem.
%
Basically, the lower triangle is (a direct generalization of) \cite[Th.~5.12]{kataoka2}.
The other triangles are obtained by applying main results of \cite{bss}.

Note that we have an advantage in using $\Pi$ instead of $\Theta_{T, L_{\infty}}$ in Remark \ref{rem imc1}; $\Pi$ has obvious counterparts $\Pi_{m, n}$ for zero-dimensional ring $R_{m, n}$, for which the above diagram can be constructed.

In the rest of this subsection, fixing $m \geq 1$ and $n \geq 0$,
we put
\begin{equation}\label{eq:abbrev}
R = R_{m, n}, \qquad A = T_{m, n}.
\end{equation}
(Note that $A$ denoted the quotient field of $\cA$ in \S \ref{sec euler}, but it will never appear in the following and there will be no danger of confusion.)
\subsubsection{Stark systems}\label{subsubsec:SS}

Let us define the module of Stark systems and then construct the lower triangle of \eqref{eq:diagram}.
We closely follow \cite[\S 5]{kataoka2}.

We set
\[
\PP(A) = \{ \fq \not\in S \mid \text{$\fq$ splits completely in $K_{p^m}$ and $A/(\Frob_{\fq}-1)A$ is free of rank one over $R$} \}.
\]
This coincides with the set $\PP$ in \cite[\S 3.1.2]{bss}.
Then for each $\fq \in \PP(A)$, we have submodules
\[
H^1_{\f}(K_{\fq}, A), H^1_{\tr}(K_{\fq}, A) \subset H^1(K_{\fq}, A),
\]
which are free of rank one, and $H^1(K_{\fq}, A) = H^1_{\f}(K_{\fq}, A) \oplus H^1_{\tr}(K_{\fq}, A)$ (see \cite[\S 3.1.3]{bss}).

For a subset $\cQ \subset \PP(A)$, we define $\cN(\cQ)$ as the set of square-free products of elements of $\cQ$.
For $\fn \in \cN(\PP(A))$, we write $\nu(\fn)$ for the number of the prime divisors of $\fn$.
By convention, for the unit ideal $(1)$, we have $(1) \in \cN(\cQ)$ for any $\cQ$ and $\nu((1)) = 0$.

\begin{definition}[{\cite[Def.~5.2]{kataoka2}}]\label{defn:large}
We say that $\fn \in \cN(\PP(A))$ is large (for $A$) if the natural localization map
\[
H^1(\cO_{K, S}, A^{\vee}(1)) \to \bigoplus_{\fq \mid \fn} H^1_f(K_{\fq}, A^{\vee}(1))
\]
is injective.
\end{definition}

For each $\fn \in \cN(\cP(A))$, let us put $S^{\fn} = S \cup \prim(\fn)$, where $\prim(\fn)$ denotes the set of prime divisors of $\fn$.
Then, since we have an exact sequence
\[
0 \to H^1(\cO_{K, S}, A^{\vee}(1)) \to H^1(\cO_{K, S^{\fn}}, A^{\vee}(1)) 
\to \bigoplus_{\fq \mid \fn} H^1_{/\f}(K_{\fq}, A^{\vee}(1)),
\]
we see that $\fn$ is large if and only if the map
\[
H^1(\cO_{K, S^{\fn}}, A^{\vee}(1)) \to \bigoplus_{\fq \mid \fn} H^1(K_{\fq}, A^{\vee}(1))
\]
is injective.

The following is a consequence of the Chebotarev density theorem.
We make essential use of Hypothesis \ref{hyp:1} here.

\begin{lemma}[{\cite[Lem.~3.9]{bss}}]\label{lem:large}
Suppose Hypothesis \ref{hyp:1} holds.
Then there exist infinitely many elements $\fn \in \cN(\PP(A))$ which are large for $A$.
\end{lemma}

%

\begin{proposition}[{\cite[Prop.~5.8]{kataoka2}}]\label{prop:des}
Suppose that Hypotheses \ref{hyp:nonanom} and \ref{hyp:2} hold.
Let $\fn \in \cN(\PP(A))$ be large.
Then $H^1(\cO_{K, S^{\fn}}, A)$ is a free $R$-module of rank $r + \nu(\fn)$ and we have a quasi-isomorphism
\begin{equation}\label{quasi}
\RG(\cO_{K, S}, A) 
\simeq \left[ H^1(\cO_{K, S^{\fn}}, A) \to \bigoplus_{\fq \mid \fn} H^1_{/\f}(K_{\fq}, A) \right],
\end{equation}
where the right hand side is regarded as a complex concentrated in degrees one and two.
\end{proposition}

\begin{proof}
We shall first show that $H^1(\cO_{K,S^\fn},A)$ is free of rank $r + \nu(\fn)$.
By Hypothesis \ref{hyp:2}, we have $H^0(\cO_{K, S^{\fn}}, A) = 0$.
Since $\RG(\cO_{K, S^{\fn}}, A)$ is a perfect complex with Euler characteristic $r$, it is enough to show that $H^2(\cO_{K, S^{\fn}}, A)$ is free of rank $\nu(\fn)$.

We use the Poitou-Tate duality:
\begin{align*}
& H^1(\cO_{K, S^{\fn}}, A)
\to \bigoplus_{\fq \in S^{\fn} \setminus S_{\infty}(K)} H^1(K_{\fq}, A)
\overset{*}{\to} H^1(\cO_{K, S^{\fn}}, A^{\vee}(1))^{\vee}\\
\to & H^2(\cO_{K, S^{\fn}}, A)
\to \bigoplus_{\fq \in S^{\fn} \setminus S_{\infty}(K)} H^2(K_{\fq}, A)
\to H^0(\cO_{K, S^{\fn}}, A^{\vee}(1))^{\vee} .
\end{align*}
By the assumption that $\fn$ is large, the map with $*$ is surjective.
Also, by Hypothesis \ref{hyp:2}, the last term vanishes.
Hence, by Hypothesis \ref{hyp:nonanom}, we obtain an isomorphism
\begin{equation}\label{h2isom}
H^2(\cO_{K, S^{\fn}}, A) \xrightarrow{\sim} \bigoplus_{\fq \mid \fn}H^2(K_\fq,A),
\end{equation}
which shows that $H^2(\cO_{K,S^\fn},A)$ is free of rank $\nu(\fn)$.

To complete the proof, we need to show the quasi-isomorphism \eqref{quasi}. 
But this follows immediately from \eqref{h2isom} and the natural long exact sequence
\begin{multline*}
0\to H^1(\cO_{K,S},A)\to H^1(\cO_{K,S^\fn},A)\to \bigoplus_{\fq \mid \fn}H^1_{/f}(K_\fq,A)\\
\to H^2(\cO_{K,S},A) \to H^2(\cO_{K,S^\fn},A)\to \bigoplus_{\fq \mid \fn }H^2(K_\fq,A).
\end{multline*}
\end{proof}


We shall review the definition of the module of Stark systems. 

\begin{definition}[{\cite[\S 3.2]{sbA}, \cite[\S 4.1]{bss}, \cite[Def.~5.4 and 5.5]{kataoka2}}]\label{defn:SS}
Let $\cQ \subset \PP(A)$ be a subset.
For each $\fn \in \cN(\cQ)$, we put
\[
X_{\fn}^r(A) = \left( {\bigcap}_R^{r + \nu(\fn)} H^1(\cO_{K, S^{\fn}}, A) \right) 
\otimes_R {\det}_{R}^{-1} \left( \bigoplus_{\fq \mid \fn} H^1_{/\f}(K_{\fq}, A) \right).
\]
For each $\fn \mid \fn'$, by the exact sequence
\[
0 \to H^1(\cO_{K, S^{\fn}}, A) \to H^1(\cO_{K, S^{\fn'}}, A) \to \bigoplus_{\fq \mid \fn'/\fn} H^1_{/\f}(K_{\fq}, A),
\]
we have a natural homomorphism $X_{\fn'}^r(A) \to X_{\fn}^r(A)$.
With respect to these transition maps, we define
\[
\StS_r(A, \cQ) = \varprojlim_{\fn \in \cN(\cQ)} X_{\fn}^r(A).
\]
\end{definition}


We shall obtain the lower triangle in the diagram \eqref{eq:diagram}.

\begin{theorem}[{\cite[Th.~5.12]{kataoka2}}]\label{thm:SSfree}
Suppose that Hypotheses \ref{hyp:nonanom}, \ref{hyp:2}, and \ref{hyp:1} hold.
Then we have a commutative diagram
\[
\xymatrix{
	\StS_{r}(A, \PP(A)) \ar[r] \ar[d]^{\simeq}
	& X_{(1)}^{r}(A) \ar@{=}[d]\\
	{\det}_R^{-1}(\RG(\cO_{K, S}, A)) \ar[r]_-{\Pi_{m, n}}
	& {\bigcap}_R^{r} H^1(\cO_{K, S}, A).
}
\]
Here, $(1)$ denotes the unit ideal.
The upper horizontal map is the canonical projection.
The right vertical equality is by definition.
\end{theorem}

\begin{proof}
The left vertical isomorphism follows from Lemma \ref{lem:large} and Proposition \ref{prop:des}.
The commutativity is a consequence of the actual constructions of maps.
\end{proof}

\subsubsection{Kolyvagin systems}\label{subsubsec:KS}

We introduce the module of Kolyvagin systems and construct the other two triangles of \eqref{eq:diagram}.

First we introduce a Selmer structure $\cF$ on the Galois representation $A = T_{m, n}$ (see \cite[\S 3.1.1]{bss} for a general definition of Selmer structures)
 by
\[
H^1_{\cF}(K_{\fq}, A) = 
\begin{cases}
H^1(K_{\fq}, A) & (\fq \in S \setminus S_{\infty}(K))\\
H^1_{f}(K_{\fq}, A) & (\fq \not \in S).
\end{cases}
\]
Then the Selmer group $H^1_{\cF}(K, A)$ is identified with $H^1(\cO_{K, S}, A)$.
More generally, for $\fn \in \cN(\cP(A))$, we have
\[
H^1_{\cF^{\fn}}(K, A) = H^1(\cO_{K, S^{\fn}}, A)
\]
and
\[
H^1_{\cF(\fn)}(K, A) 
= \Ker \left( H^1(\cO_{K, S^{\fn}}, A) \to \bigoplus_{\fq \mid \fn} H^1_{/\tr}(K_{\fq}, A) \right)
\]
(see \cite[\S 3.1.3]{bss} for the definitions of the modified Selmer structures $\cF^{\fn}$ and $\cF(\fn)$).

For each $\fq \in \PP(A)$, we put $G_{\fq} = \Gal(K(\fq)/K(1))$ (recall that $K(\fq)$ denotes the maximal $p$-extension in the $\fq$-ray class field of $K$).
More generally, for $\fn \in \cN(\cP(A))$, we put $G_{\fn} = \bigotimes_{\fq \mid \fn} G_{\fq}$.

\begin{definition}[{\cite[Def.~4.1]{sbA}, \cite[\S 5.1]{bss}}]\label{defn:KS}
Let $\cQ \subset \PP(A)$ be a subset.
We define a Kolyvagin system
\[
\kappa = (\kappa_{\fn})_{\fn} \in \prod_{\fn \in \cN(\cQ)} 
{\bigcap}_R^r H^1_{\cF(\fn)}(K, A) \otimes G_{\fn}
\]
by requiring the ``finite-singular relation'' (we do not recall the precise definition).
Let $\KS_r(A, \cQ)$ denote the module of Kolyvagin systems.
\end{definition}


Now we consider the middle triangle of \eqref{eq:diagram}.

\begin{definition}[{\cite[\S 4.2]{sbA}, \cite[\S 5.2]{bss}}]
Let $\cQ \subset \PP(A)$ be a subset.
We define the regulator map
\[
\Reg_r: \StS_{r}(A, \cQ) \to \KS_{r}(A, \cQ)
\]
as follows.
For each $\fn \in \cN(\cQ)$, by the exact sequence defining $H^1_{\cF(\fn)}(K, A)$, we have a natural map
\[
{\bigcap}_R^{r + \nu(\fn)} H^1(\cO_{K, S^{\fn}}, A) 
\to {\bigcap}_R^{r} H^1_{\cF(\fn)}(K, A) \otimes \bigotimes_{\fq \mid \fn} H^1_{/\tr}(K_{\fq}, A). 
\]
By combining with the ``finite-singular comparison map,'' we then obtain a map
\[
X_{\fn}^r(A) \to {\bigcap}_R^r H^1_{\cF(\fn)}(K, A) \otimes G_{\fn}.
\]
These maps for various $\fn$ define $\Reg_r$.
\end{definition}

By the construction, we obtain the middle commutative triangle in \eqref{eq:diagram}.
Then we have to show that $\Reg_{r}$ is an isomorphism:

\begin{theorem}[{\cite[Th.~5.2(i)]{bss}}]\label{thm:SStoKS}
Suppose that Hypotheses \ref{hyp:nonanom}, \ref{hyp:2}, and \ref{hyp:1} hold.
Suppose $p \geq 5$.
Then the regulator map
\[
\Reg_{r}: \StS_{r}(A, \PP(A)) \overset{\sim}{\to} \KS_{r}(A, \PP(A))
\]
is an isomorphism.
\end{theorem}

\begin{proof}
In order to apply \cite[Th.~5.2(i)]{bss}, we have to check \cite[Hyp.~4.2]{bss}, i.e., 
that there exist infinitely many elements $\fn \in \cN(\PP(A))$ such that 
\[
H^1_{(\cF^*)_{\fn}}(K, A^{\vee}(1)) = 0
\]
 and that $H^1_{\cF^{\fn}}(K, A)$ is free of rank $r + \nu(\fn)$ over $R$.
By Lemma \ref{lem:large}, it is enough to show that every $\fn \in \cN(\PP(A))$ which is large for $A$ (in the sense of Definition \ref{defn:large}) satisfies the conditions.
By Proposition \ref{prop:des}, the module $H^1_{\cF^{\fn}}(K, A) = H^1(\cO_{K, S^{\fn}}, A)$ is actually free of rank $r+\nu(\fn)$.
Moreover, since $H^1_{(\cF^*)_{\fn}}(K, A^{\vee}(1))$ is the kernel of
\[
H^1(\cO_{K, S}, A^{\vee}(1)) \to \bigoplus_{\fq \in S^{\fn} \setminus S_{\infty}(K)} H^1(K_{\fq}, A^{\vee}(1)),
\]
it vanishes as $\fn$ is large.
\end{proof}



Finally we consider the upper triangle of \eqref{eq:diagram}.

\begin{theorem}[{\cite[Cor.~6.13]{bss}}]\label{thm:EStoKS}
Let $\cK$ be an abelian extension of $K$ such that $\cK \supset L_{\infty}$ and $\cK/L$ is a pro-$p$ extension.
Suppose that $\cK \supset K(\fq)$ for every $\fq \in \PP(A)$.
Suppose that $H^0(L, T/p) = 0$ and Hypothesis \ref{hyp:6} hold.
Then we have a natural homomorphism
\[
\cD_r: \ES_r(T, \cK) \to \KS_r(A, \PP(A)),
\]
called the derivative operator, such that we have a commutative diagram
\[
\xymatrix{
	\ES_{r}(T, \cK) \ar[rd] \ar[d]_{\cD_r}
	& \\
	\KS_{r}(A, \PP(A)) \ar[r]
	& {\bigcap}_{R}^{r} H^1(\cO_{K, S}, A),
}
\]
where the diagonal and the horizontal maps are the natural projection maps.
\end{theorem}

\begin{proof}
We only have to apply \cite[Cor.~6.13]{bss}, taking the following remarks into account.
In the corollary, the Selmer structure is taken as the canonical one $\cF_{\can}$, but that does not matter since our Selmer structure $\cF$ is larger.
We also removed the latter condition in \cite[Hyp.~6.7]{bss}, that is,
``$\cK$ contains a $\Z_p^d$-extension of $K$ ($d \geq 1$) in which no finite place of $K$ splits completely.''
For, this condition is only used in order to check that the Kolyvagin derivative of an Euler system satisfies the unramified condition outside $S^{\fn}$, which is already assumed in our definition of Euler systems.
\end{proof}

Thus we have the upper triangle in \eqref{eq:diagram}.
This completes the construction of the commutative diagram \eqref{eq:diagram}, so we have also finished proving Theorem \ref{thm:goal} and Theorem \ref{main}.

\section{Derivatives of Euler systems}\label{sec der}

In this section, we generalize several conjectures and results in \cite{bks4}, where the case of elliptic curves over $\Q$ is considered, to a general motive. 


We use notations in \S \ref{sec euler}. Throughout this section, we assume Hypothesis \ref{hypint}. 

\subsection{Bockstein maps}\label{sec boc}

As in \S \ref{sec ext}, we set
$$e:={\rm dim}_A(H^2(\cO_{K,S},V))={\rm rank}_\cA(H_\Sigma^2(\cO_{K,S},T)_{\rm tf})$$
and fix an $\cA$-basis $x \in {\bigwedge}_\cA^e H^2_\Sigma(\cO_{K,S},T)_{\rm tf}$. 
In this subsection, we define a ``Bockstein regulator map"
$${\rm Boc}_F={\rm Boc}_{T,F,x}: {\bigwedge}_\cA^{r+e} H^1_\Sigma(\cO_{K,S},T) \to {\bigwedge}_{\cA}^r H^1_\Sigma(\cO_{K,S},T)\otimes_\cA I_F^e/I_F^{e+1}$$
for each $F \in \Omega(\cK)$, where $I_F$ denotes the augmentation ideal of $\cA[\cG_F]$:
$$I_F:=\ker(\cA[\cG_F] \twoheadrightarrow \cA).$$
The definition of the map is more or less the same as that in \cite[\S 2.3]{bks4}. 

Let
$$\beta=\beta_F: H^1_\Sigma(\cO_{K,S},T) \to H^2_\Sigma(\cO_{K,S},T)_{\rm tf}\otimes_\cA  I_F/I_F^2$$
be the Bockstein map, i.e., the map induced by the connecting homomorphism of the natural exact triangle
$$\rgamma_\Sigma(\cO_{K,S},T)\lotimes_\cA I_F/I_F^2 \to \rgamma_\Sigma(\cO_{F,S},T)\lotimes_{\cA[\cG_F]}\cA[\cG_F]/I_F^2 \to \rgamma_\Sigma(\cO_{K,S},T)\to.$$
Write $x=x_1\wedge\cdots \wedge x_e$ and define $\beta_i: H^1_\Sigma(\cO_{K,S},T) \to I_F/I_F^2$ by
$$\beta(a)=\sum_{i=1}^e x_i\otimes \beta_i(a).$$
The Bockstein regulator map is now defined by
$${\rm Boc}_F(a_1\wedge\cdots \wedge a_{r+e}) = \sum_{\sigma }{\rm sgn}(\sigma)a_{\sigma(e+1)}\wedge\cdots \wedge a_{\sigma(e+r)}\otimes  \det(\beta_i(a_{\sigma(j)}))_{1\leq i,j\leq e} ,$$
where $\sigma$ runs over the elements of the symmetric group $S_{r+e}$ such that $\sigma(1)<\cdots <\sigma(e)$ and $\sigma(e+1)<\cdots < \sigma(e+r)$. 
Note that ${\rm Boc}_F$ depends on the choice of $x$, but is independent of $x_1,\ldots,x_e$. 

\subsection{Derivatives and extended special elements}\label{sec der finite}

We now assume that a canonical Euler system
$$c \in {\rm ES}_r(T,\cK)$$
is given. In this subsection, we formulate a conjecture which relates ``derivatives" of $c$ with the extended special element $\widetilde \eta_K$ defined in \S \ref{sec ext}. (Here ``derivatives" are different from Kolyvagin's derivatives, which are considered in the usual Euler system argument \cite{R} as in Theorem \ref{thm:EStoKS}, but similar to those considered by Darmon \cite{DH}.) This conjecture is a generalization of the ``generalized Perrin-Riou conjecture" for Kato's Euler system formulated in \cite[Conj.~2.12]{bks4}. 

As in Conjecture \ref{conjes}, we fix an $\cA[[\Gal(\cK/K)]]$-basis
$$b=(b_F)_F \in \varprojlim_{F \in \Omega(\cK)} {\bigwedge}_{\cA[\cG_F]}^r Y_F(T^\ast(1))^\ast.$$
(As Conjecture \ref{conjes} suggests, the Euler system $c$ should depend on the choice of $b$.) 

We fix $F \in \Omega(\cK)$, which is unramified outside $S$ (so that $S(F)=S$). 
Let
\begin{eqnarray*}
\iota_F: {\bigwedge}_\cA^r H^1_\Sigma(\cO_{K,S},T) \otimes_\cA I_F^e/I_F^{e+1} &\hookrightarrow& {\bigcap}_{\cA[\cG_F]}^r H_\Sigma^1(\cO_{F,S},T) \otimes_\cA I_F^e/I_F^{e+1}\\
&\subset &{\bigcap}_{\cA[\cG_F]}^r H_\Sigma^1(\cO_{F,S},T) \otimes_\cA \cA[\cG_F]/I_F^{e+1}
\end{eqnarray*}
be the canonical injection defined in the same way as \cite[Lem.~2.11]{sano}. (Note that ${\bigwedge}_\cA^r H^1_\Sigma(\cO_{K,S},T) ={\bigcap}_\cA^r H^1_\Sigma(\cO_{K,S},T)$ since the $\cA$-module $H^1_\Sigma(\cO_{K,S},T)$ is free.)

Let
$$\widetilde \eta_K=\widetilde \eta_{K,S,\Sigma}(T) \in \CC_p \otimes_{\ZZ_p} {\bigwedge}_\cA^{r+e} H^1(\cO_{K,S},T)$$
be the extended special element in Definition \ref{def ext} (which depends on the fixed $\cA$-bases $x\in {\bigwedge}_\cA^e H^2_\Sigma(\cO_{K,S},T)_{\rm tf}$ and $b_K \in {\bigwedge}_\cA^r Y_K(T^\ast(1))^\ast$). 

The following conjecture predicts a precise relation between the element
$$\cN_{F/K}(c_F):=\sum_{\sigma \in \cG_F} \sigma c_F \otimes \sigma^{-1} \in {\bigcap}_{\cA[\cG_F]}^r H^1_\Sigma(\cO_{F,S},T) \otimes_\cA \cA[\cG_F]$$
and $\widetilde \eta_K$. 

\begin{conjecture}\label{der formula}
Assume Hypothesis \ref{hypint} for $F$. 
\begin{itemize}
\item[(i)] We have
$$\cN_{F/K}(c_F) \in  {\bigcap}_{\cA[\cG_F]}^r H^1_\Sigma(\cO_{F,S},T) \otimes_\cA I_F^e. $$
\item[(ii)] Assume the integrality of $\widetilde \eta_K$, i.e., 
$$\widetilde \eta_K \in {\bigwedge}_\cA^{r+e} H^1_\Sigma(\cO_{K,S},T).$$
(This is a consequence of the Tamagawa number conjecture: see Proposition \ref{tnc ext}.) Then we have
$$\cN_{F/K}(c_F) =(-1)^{re}\iota_F\left( {\rm Boc}_F(\widetilde \eta_K)\right) \text{ in } {\bigcap}_{\cA[\cG_F]}^r H^1_\Sigma(\cO_{F,S},T) \otimes_\cA I_F^e/I_F^{e+1},$$
where 
$${\rm Boc}_F={\rm Boc}_{T,F,x}: {\bigwedge}_\cA^{r+e} H^1_\Sigma(\cO_{K,S},T) \to {\bigwedge}_{\cA}^r H^1_\Sigma(\cO_{K,S},T)\otimes_\cA I_F^e/I_F^{e+1}$$
is the Bockstein regulator map defined in \S \ref{sec boc}. 
\end{itemize}
\end{conjecture}

\begin{remark}\label{rem dar}
Conjecture \ref{der formula}(ii) in particular predicts the existence of a unique element
$$\kappa_F \in  {\bigwedge}_{\cA}^r H^1_\Sigma(\cO_{K,S},T)\otimes_\cA I_F^e/I_F^{e+1}$$
such that
$$\iota_F(\kappa_F) = \cN_{F/K}(c_F)\text{ in }{\bigcap}_{\cA[\cG_F]}^r H_\Sigma^1(\cO_{F,S},T) \otimes_\cA \cA[\cG_F]/I_F^{e+1}.$$
We call this conjectural element the {\it Darmon derivative} of the Euler system $c_F$. Conjecture \ref{der formula}(ii) is then equivalent to the equality
$$\kappa_F=(-1)^{re} {\rm Boc}_F(\widetilde \eta_K). $$
\end{remark}

\begin{remark}
When $e=0$ (i.e., $H^2(\cO_{K,S},V)=0$), Conjecture \ref{der formula}(i) is trivially true and (ii) is equivalent to the equality
$$c_K=\eta_K,$$
where $\eta_K\in  \CC_p\otimes_{\ZZ_p}{\bigwedge}_\cA^r H^1(\cO_{K,S},T)$ is the special element in Definition \ref{special} (in this case we do not need to assume the integrality). If $K=\QQ$, $M=h^1(E)(1)$ with an elliptic curve $E$ over $\QQ$ and $c=z^{\rm Kato}$, then Conjecture \ref{der formula} is equivalent to \cite[Conj.~2.12]{bks4} and the equality $z_\QQ^{\rm Kato}=\eta_\QQ$ is equivalent to Perrin-Riou's conjecture (see \S \ref{kato ex}).
\end{remark}

\begin{remark}
In the $\GG_m$ case (i.e., $M=h^0(K)(1)$ and $c=\eta^{\rm RS}$: see \S \ref{rs ex}), Conjecture \ref{der formula} is equivalent to the ($p$-part of the) ``Mazur-Rubin-Sano conjecture" formulated in \cite[Conj.~5.2]{MRGm} and \cite[Conj.~3]{sano}. More precisely, Conjecture \ref{der formula} is equivalent to ${\rm MRS}(F/K/K, S, \Sigma, S_\infty(K), S\setminus \{v_0\})_p$ in \cite[Conj.~4.2]{bks2}, where we choose a finite place $v_0 \in S$. (This choice corresponds to the choice of a $\ZZ_p$-basis $x \in {\bigwedge}_{\ZZ_p}^e H_\Sigma^2(\cO_{K,S},\ZZ_p(1))_{\rm tf} $. See Remark \ref{rem cnf}.) 
\end{remark}

\begin{remark}
One can show that the equivariant Tamagawa number conjecture for $(M\otimes_{K} F,\cA[\cG_F])$ (see \cite[Conj.~4]{BFetnc}) implies Conjecture \ref{der formula} when $c_F$ is the special element $\eta_F$ in Definition \ref{special} (see also Remark \ref{rem etnc}). 
\end{remark}

Concerning the existence of Darmon derivatives in Remark \ref{rem dar}, we have the following result. 

\begin{proposition}\label{prop dar}
Let $L_\infty/K$ be an extension as in \S \ref{formulate imc} and $c_{L_\infty} \in {\bigcap}_\Lambda^rH^1_\Sigma(\cO_{K,S},\TT)$ an Euler system. Assume Hypothesis \ref{hypint} for $\cK=L_\infty$. If there exists an element $\fz_{L_\infty} \in {\det}_\Lambda^{-1}(\rgamma_\Sigma(\cO_{K,S},\TT))$ such that the map \eqref{canisom2} sends $\fz_{L_\infty}$ to $c_{L_\infty}$, then the Darmon derivative $\kappa_F$ of $c_F$ exists for any $F \in \Omega(L_\infty)$. 
\end{proposition}

\begin{proof}
Let 
$$\fz_F \in {\det}_{\cA[\cG_F]}^{-1}(\rgamma_\Sigma(\cO_{F,S},T))$$
be the image of $\fz_{L_\infty}$ under the natural surjection ${\det}_\Lambda^{-1}(\rgamma_\Sigma(\cO_{K,S},\TT)) \twoheadrightarrow {\det}_{\cA[\cG_F]}^{-1}(\rgamma_\Sigma(\cO_{F,S},T))$. 

We have the following commutative diagram:
\begin{multline}\label{diag}
\xymatrix{
{\det}_{\cA[\cG_F]}^{-1}(\rgamma_\Sigma(\cO_{F,S},T)) \ar[r]^-{\Theta_{T,F}} \ar@{->>}[d] & {\bigcap}_{\cA[\cG_F]}^r H^1_\Sigma(\cO_{F,S},T)  \ar[r]^-{\cN_{F/K}} &  {\bigcap}_{\cA[\cG_F]}^r H_\Sigma^1(\cO_{F,S},T) \otimes_\cA \cA[\cG_F]/I_F^{e+1} \\ 
{\det}_\cA^{-1}(\rgamma_\Sigma(\cO_{K,S},T)) \ar[r]_-{\Theta_x} &  {\bigwedge}_\cA^{r+e} H^1_\Sigma(\cO_{K,S},T) \ar[r]_-{(-1)^{re}{\rm Boc}_F} & {\bigwedge}_\cA^r H^1_\Sigma(\cO_{K,S},T) \otimes_\cA I_F^e/I_F^{e+1}. \ar@{^{(}->}[u]_{\iota_{F}}
}
\end{multline}
Here $\Theta_{T,F}$ is the $F$-component of the map $\Theta_{T,L_\infty}$ mentioned in Remark \ref{rem imc1}, and $\Theta_x$ is the following map:
\begin{eqnarray*}
{\det}_\cA^{-1}(\rgamma_\Sigma(\cO_{K,S},T)) &\simeq& {\det}_\cA(H^1_\Sigma(\cO_{K,S},T)) \otimes_\cA {\det}_\cA^{-1}(H^2_\Sigma(\cO_{K,S},T)) \\
&\simeq & {\rm Fitt}_\cA(H^2_\Sigma(\cO_{K,S},T)_{\rm tors}) \cdot {\bigwedge}_\cA^{r+e}H^1_\Sigma(\cO_{K,S},T) \otimes_\cA {\bigwedge}_\cA^{e} H^2_\Sigma(\cO_{K,S},T)_{\rm tf}^\ast\\
&\simeq &{\rm Fitt}_\cA(H^2_\Sigma(\cO_{K,S},T)_{\rm tors}) \cdot {\bigwedge}_\cA^{r+e}H^1_\Sigma(\cO_{K,S},T) \\
&\subset &{\bigwedge}_\cA^{r+e}H^1_\Sigma(\cO_{K,S},T),
\end{eqnarray*}
where the third isomorphism is defined by using the basis $x \in {\bigwedge}_\cA^e H^2_\Sigma(\cO_{K,S},T)_{\rm tf}$. The commutativity of the diagram is proved in the same way as \cite[Lem.~5.22]{bks1} or \cite[Th.~7.8]{bks4}. 

Since we have $\Theta_{T,F}(\fz_F) = c_F$, the commutative diagram implies that the element $\cN_{F/K}(c_F)$ lies in the image of $\iota_F$. This shows the existence of the Darmon derivative. 
\end{proof}


\subsection{An Iwasawa theoretic version}\label{der iwasawa section}

As in \cite[\S 4.3]{bks4}, we can formulate a natural Iwasawa theoretic version of Conjecture \ref{der formula}. 

Let $K_\infty/K$ be a $\ZZ_p$-extension and consider the case $\cK=K_\infty$. We keep assuming Hypothesis \ref{hypint}. Note that each $F \in \Omega(K_\infty)$ is of the form $K_n$ (the $n$-th layer) for some $n$. We set
$$I_n:=I_{K_n} \text{ and }I:=\ker (\cA[[\Gal(K_\infty/K)]] \twoheadrightarrow \cA)\simeq \varprojlim_n I_n.$$
We define the Bockstein regulator map for $K_\infty$ by
\begin{eqnarray*}
{\rm Boc}_\infty:=\varprojlim_n {\rm Boc}_{K_n}: {\bigwedge}_\cA^{r+e} H^1_\Sigma(\cO_{K,S},T) &\to& {\bigwedge}_{\cA}^r H^1_\Sigma(\cO_{K,S},T)\otimes_\cA  \varprojlim_n I_n^e/I_n^{e+1}\\
&\simeq&  {\bigwedge}_{\cA}^r H^1_\Sigma(\cO_{K,S},T)\otimes_\cA I^e/I^{e+1}.
\end{eqnarray*}
This map induces
\begin{eqnarray*}
{\rm Boc}_\infty:\CC_p\otimes_{\ZZ_p} {\bigwedge}_\cA^{r+e} H^1_\Sigma(\cO_{K,S},T) \to   \CC_p\otimes_{\ZZ_p}{\bigwedge}_{\cA}^r H^1_\Sigma(\cO_{K,S},T)\otimes_\cA I^e/I^{e+1}.
\end{eqnarray*}
(Note that $I^e/I^{e+1}\simeq \cA$ and so it does not vanish after taking $\CC_p\otimes_{\ZZ_p}-$.) 

For a given canonical Euler system $c \in {\rm ES}_r(T,K_\infty)$, we assume the existence of the Darmon derivative of $c$
$$\kappa_n:=\kappa_{K_n} \in {\bigwedge}_{\cA}^r H^1_\Sigma(\cO_{K,S},T)\otimes_\cA I_n^e/I_n^{e+1}$$
in Remark \ref{rem dar} for every $n$. Then one sees that $(\kappa_n)_n$ is an inverse system and so we can define the limit
$$\kappa_\infty:=\varprojlim_n \kappa_n \in {\bigwedge}_{\cA}^r H^1_\Sigma(\cO_{K,S},T)\otimes_\cA \varprojlim_n I_n^e/I_n^{e+1}\simeq{\bigwedge}_{\cA}^r H^1_\Sigma(\cO_{K,S},T)\otimes_\cA I^e/I^{e+1}.$$

We now propose the following conjecture.

\begin{conjecture}\label{der iw}
We have
$$\kappa_\infty=(-1)^{re}{\rm Boc}_\infty(\widetilde \eta_K) \text{ in }\CC_p\otimes_{\ZZ_p}{\bigwedge}_{\cA}^r H^1_\Sigma(\cO_{K,S},T)\otimes_\cA I^e/I^{e+1}.$$
\end{conjecture}

\begin{remark}
Unlike Conjecture \ref{der formula}, we do not need to assume the integrality of $\widetilde \eta_K$ in Conjecture \ref{der iw}. This is one of the advantages of taking limits. 
\end{remark}

\begin{remark}
Conjecture \ref{der iw} is a generalization of \cite[Conj.~4.2]{bks2} and \cite[Conj.~4.9]{bks4}. In particular, by \cite[Th.~4.9]{bks2} and \cite[Cor.~6.7]{bks4}, Conjecture \ref{der iw} constitutes a generalization of both the Gross-Stark conjecture \cite{Gp} and the $p$-adic Birch-Swinnerton-Dyer conjecture \cite{MTT}. 
\end{remark}

We shall show that Conjectures \ref{der formula} and \ref{der iw} are equivalent under suitable assumptions. 

\begin{proposition}\label{der equivalent}
Let $L_\infty/K$ be an extension as in \S \ref{formulate imc} (with $d=1$) and $c_{L_\infty} \in {\bigcap}_\Lambda^rH^1_\Sigma(\cO_{K,S},\TT)$ an Euler system. We assume Hypothesis \ref{hypint} for $\cK=L_\infty$. We also assume the following.
\begin{itemize}
\item[(i)] There exists an element $\fz_{L_\infty} \in {\det}_\Lambda^{-1}(\rgamma_\Sigma(\cO_{K,S},\TT))$ such that the map \eqref{canisom2} sends $\fz_{L_\infty}$ to $c_{L_\infty}$.
\item[(ii)] The map ${\rm Boc}_\infty$ is non-zero.
\end{itemize}
 Then Conjecture \ref{der formula} for any $F \in \Omega(L_\infty)$ is equivalent to Conjecture \ref{der iw}. 
\end{proposition}

\begin{proof}
It is obvious that Conjecture \ref{der formula} (for $K_n$ for all $n$) implies Conjecture \ref{der iw}, so we shall prove the converse. 

Let 
$$\fz_K \in {\det}_\cA^{-1}(\rgamma_\Sigma(\cO_{K,S},T))$$
be the image of $\fz_{L_\infty}$.  Let
$$\Theta_x: {\det}_\cA^{-1}(\rgamma_\Sigma(\cO_{K,S},T)) \to  {\bigwedge}_\cA^{r+e}H^1_\Sigma(\cO_{K,S},T)$$
be the map defined in the proof of Proposition \ref{prop dar}. 
By the commutative diagram (\ref{diag}) (for $F=K_n$ for all $n$), we see that the Darmon derivative of $c_{K_\infty}$ exists and it is given by
$$\kappa_\infty = (-1)^{re} {\rm Boc}_\infty(\Theta_x(\fz_K)).$$
Since ${\rm Boc}_\infty$ is non-zero by assumption, we see that Conjecture \ref{der iw} is equivalent to the equality
\begin{equation}\label{theta equality}
\Theta_x(\fz_K)=\widetilde \eta_K.
\end{equation}
Note that, since $\Theta_x(\fz_K)$ lies in ${\bigwedge}_\cA^{r+e}H^1_\Sigma(\cO_{K,S},T)$, this equality implies the integrality of $\widetilde \eta_K$, which is assumed in Conjecture \ref{der formula}(ii). 

Let $F \in \Omega(L_\infty)$. By the proof of Proposition \ref{prop dar}, the Darmon derivative of $c_F$ is given by
$$\kappa_F=(-1)^{re}{\rm Boc}_F(\Theta_x(\fz_K)).$$
So Conjecture \ref{der formula} for $F$ is equivalent to the equality
$${\rm Boc}_F(\Theta_x(\fz_K)) = {\rm Boc}_F(\widetilde \eta_K).$$
This is obviously implied by \eqref{theta equality}. Thus we have proved that Conjecture \ref{der iw} implies Conjecture \ref{der formula} for any $F\in \Omega(L_\infty)$. 
\end{proof}

\subsection{A strategy for proving the Tamagawa number conjecture}

As in the previous subsection, we consider the case $\cK=K_\infty$. 

\begin{theorem}\label{descent}
Assume Hypotheses \ref{hypint} and \ref{leop}. If we also assume
\begin{itemize}
\item Conjecture \ref{neIMC} (the non-equivariant Iwasawa main conjecture) for $c$,
\item Conjecture \ref{der iw} for $c$, and 
\item the map ${\rm Boc}_\infty$ is non-zero, 
\end{itemize}
then the Tamagawa number conjecture for $M^\ast(1)$ (with coefficients in $\cA$) is true. 
\end{theorem}

\begin{proof}
Conjecutre \ref{neIMC} implies the existence of a $\Lambda$-basis $\fz_{K_\infty} \in {\det}_\Lambda^{-1}(\rgamma_\Sigma(\cO_{K,S},\TT))$ such that the map \eqref{canisom2} sends $\fz_{K_\infty}$ to $c_{K_\infty}$. Let
$$\fz_K \in {\det}_\cA^{-1}(\rgamma_\Sigma(\cO_{K,S},T))$$
be the image of $\fz_{K_\infty}$, which is an $\cA$-basis. By Proposition \ref{tnc ext}, it is sufficient to prove that 
\begin{equation}\label{tnc equality}
\Theta_x(\fz_K) = \widetilde \eta_K,
\end{equation}
where $\Theta_x$ is the map in the diagram \eqref{diag}. Since ${\rm Boc}_\infty$ is non-zero by assumption, this is implied by Conjecture \ref{der iw} (as in the proof of Proposition \ref{der equivalent}). So we have completed the proof. 
\end{proof}

\begin{remark}
Theorem \ref{descent} is a direct generalization of \cite[Th.~7.6]{bks4}, where a strategy for proving the Birch-Swinnerton-Dyer formula for an elliptic curve over $\QQ$ is given. 
\end{remark}

\begin{remark}
In the $\GG_m$ case, 
a natural equivariant version of Theorem \ref{descent} is given in \cite[Th.~5.2]{bks2}. Since the main aim of this paper is to study a general motive, we do not give an equivariant generalization of Theorem \ref{descent} in this paper. 
\end{remark}

\section{Heegner points}

The aim of this section is to study how the theory of Heegner points fits in the general framework given in earlier sections. In particular, we study relations between Perrin-Riou's ``Heegner point main conjecture" and the (non-equivariant) Iwasawa main conjecture in Conjecture \ref{neIMC}. 

In this section, we consider an imaginary quadratic base field $K$ and the motive $M=h^1(E/K)(1)$ (with coefficients $R=\QQ$, $A=\QQ_p$ and $\cA=\ZZ_p$), where $E$ is an elliptic curve over $\QQ$ such that $K$ satisfies the Heegner hypothesis for $E$ (i.e., every prime divisor of the conductor of $E$ splits in $K$). Let $T:=T_p(E)$ and $V:=\QQ_p\otimes_{\ZZ_p}T$. 
In this case, note that 
$$Y_K(T^\ast(1)):=\bigoplus_{v \in S_\infty(K)} H^0(K_v,T^\ast(1)) =H^0(\CC,T^\ast(1))=T^\ast(1)$$
and so
$${\rm rank}_{\ZZ_p}(Y_K(T^\ast(1))) =2,$$
i.e., the basic rank is two (see Definition \ref{def basic}). So Conjecture \ref{conjes} suggests that there should be a canonical Euler system of rank two in this setting. For this reason, we expect that {\it Heegner points are related with rank two Euler systems}. 






\subsection{Heegner elements and Euler systems of rank two}

In this subsection, we introduce ``Heegner elements", which lie in the second exterior power of $H^1$, and relate them with special elements in Definition \ref{special} (see Proposition \ref{GZ}). As an application, we construct a rank two Euler system which is related with Heegner points in the case of analytic rank one (see Theorem \ref{heeg bottom}).

We fix a finite set $S$ of places of $K$ containing $\{\infty\}\cup S_p(K)\cup S_{\rm ram}(T)$. (Here $S_\infty(K)=\{\infty\}$.) 
We need the following lemma. 

\begin{lemma}\label{lem seq}
Assume ${\rm rank}(E(K))=1$ and $\# \sha(E/K)[p^\infty]<\infty$. Then $H^2(\cO_{K,S},V)=0$ and there exists a canonical exact sequence
$$0\to \QQ_p\otimes_\ZZ E(K) \to H^1(\cO_{K,S},V) \to \QQ_p\otimes_\QQ \Gamma(E,\Omega_{E/K}^1) \to \QQ_p\otimes_\ZZ E(K)^\ast \to 0.$$
\end{lemma}

\begin{proof}
We have a canonical exact triangle
$$\rgamma_f(K,V) \to \rgamma(\cO_{K,S},V) \to \bigoplus_{v \in S}\rgamma_{/f}(K_v,V)\to.$$
(See \cite[p.~522]{BFetnc} for example.) If $v \notin S_p(K)$, we know that $\rgamma_{/f}(K_v,V)$ is acyclic. So we obtain a long exact sequence
$$ 0\to H^1_f(K,V) \to H^1(\cO_{K,S},V)\to H^1_{/f}(K_p, V) \to H^2_f(K,V) \to H^2(\cO_{K,S},V) \to 0. $$
Since we assume the finiteness of $\sha(E/K)[p^\infty]$, we have
$$H^1_f(K,V) = \QQ_p\otimes_\ZZ E(K) \text{ and }H^2_f(K,V) = \QQ_p\otimes_\ZZ E(K)^\ast.$$
Also, we have $H^1_{/f}(K_p,V) = \QQ_p\otimes_{\ZZ_p} (\varprojlim_n E(K_p)/p^n)^\ast$, which is canonically isomorphic to $\QQ_p \otimes_\QQ \Gamma(E,\Omega_{E/K}^1)$ via the dual exponential map. Hence we have an exact sequence 
\begin{multline}\label{long exact}
0\to \QQ_p\otimes_\ZZ E(K) \to H^1(\cO_{K,S},V) \to \QQ_p\otimes_\QQ \Gamma(E,\Omega_{E/K}^1) \to \QQ_p\otimes_\ZZ E(K)^\ast \to  H^2(\cO_{K,S},V)\to 0.
\end{multline}
We shall show that $H^2(\cO_{K,S},V) = 0$. Since we assume ${\rm rank}(E(K))=1$, we see that the map $\QQ_p\otimes_\ZZ E(K) \to \QQ_p\otimes_{\ZZ_p} (\varprojlim_n E(K_p)/p^n)$ is injective. This implies that the map $\QQ_p\otimes_\QQ \Gamma(E,\Omega_{E/K}^1) \to \QQ_p\otimes_\ZZ E(K)^\ast$ is surjective, so we have $H^2(\cO_{K,S},V)=0$. This completes the proof. 
\end{proof}

Throughout this subsection, we assume ${\rm rank}(E(K))=1$ and $\# \sha(E/K)[p^\infty] < \infty$. Then, by Lemma \ref{lem seq}, we obtain a canonical isomorphism
$${\det}_{\QQ_p}(H^1(\cO_{K,S},V)) \simeq {\det}_{\QQ_p}(\QQ_p\otimes_\ZZ E(K)) \otimes_{\QQ_p}{\det}_{\QQ_p}(\QQ_p\otimes_\QQ \Gamma(E,\Omega_{E/K}^1)) \otimes_{\QQ_p}{\det}_{\QQ_p}^{-1}(\QQ_p\otimes_\ZZ E(K)^\ast),$$
i.e.,
\begin{equation}\label{isom2}
{\bigwedge}_{\QQ_p}^2 H^1(\cO_{K,S},V) \simeq \QQ_p\otimes_\QQ \left( E(K)\otimes_\ZZ E(K) \otimes_\ZZ {\bigwedge}_\QQ^2 \Gamma(E,\Omega_{E/K}^1)\right).
\end{equation}

Now we fix a modular parametrization $\phi: X_0(N) \to E$ and let 
$$y_K\in E(K)$$
be the associated Heegner point. Let $c_\phi$ be the Manin constant and set $u_K:=\# \cO_K^\times/2$. We fix N\'eron differentials $\omega$ and $\omega^K$ of $E/\QQ$ and $E^K/\QQ$ respectively, where $E^K$ be the quadratic twist of $E$ by $K$. Then $\{\omega,\omega^K\}$ is a $\QQ$-basis of $\Gamma(E,\Omega_{E/K}^1)$. 
Let ${\rm Eul}_S \in \QQ^\times$ be the product of Euler factors at $v \in S\setminus \{\infty\}$ satisfying
$${\rm Eul}_S \cdot L^\ast(E/K,1)=L^\ast_S(E/K,1).$$

We now define the Heegner element. 

\begin{definition}\label{def heeg}
Assume ${\rm rank}(E(K))=1$ and $\#\sha(E/K)[p^\infty]<\infty$. Then we define the {\it Heegner element} for $E/K$
$$z_K^{\rm Hg}\in {\bigwedge}_{\QQ_p}^2 H^1(\cO_{K,S},V)$$
as the element corresponding to
$${\rm Eul}_S\cdot (u_Kc_\phi)^{-2} \otimes y_K \otimes y_K \otimes (\omega\wedge\omega^K) \in  \QQ_p\otimes_\QQ \left( E(K)\otimes_\ZZ E(K) \otimes_\ZZ {\bigwedge}_\QQ^2 \Gamma(E,\Omega_{E/K}^1)\right)$$
under the isomorphism \eqref{isom2}. 
\end{definition}

We choose a $\ZZ_p$-basis
\begin{equation}\label{bk}
b_K \in {\bigwedge}_{\ZZ_p}^2 Y_K(T^\ast(1))^\ast= {\bigwedge}_{\ZZ_p}^2 T(-1)
\end{equation}
in the following way. Let $\Omega_{E/K}$ be the N\'eron period of $E/K$. We first take a $\ZZ_{(p)}$-basis
$$\gamma^\ast \in {\bigwedge}_{\ZZ_{(p)}}^2 H_1(E(\CC),\ZZ_{(p)})^\ast$$
so that the period map
\begin{equation}\label{period map}
\RR \otimes_{\QQ} {\bigwedge}_\QQ^2 \Gamma(E,\Omega_{E/K}^1) \xrightarrow{\sim} \RR \otimes_\QQ{\bigwedge}_{\QQ}^2 H_1(E(\CC),\QQ)^\ast
\end{equation}
sends $\omega\wedge \omega^K$ to
$$ \frac{1}{\sqrt{|D_K|}} \Omega_{E/K}\cdot \gamma^\ast.$$
(This is possible, since $\frac{1}{\sqrt{|D_K|}} \Omega_{E/K}$ coincides with $\Omega_E^+ \Omega_{E^K}^+$ up to 2-power, where $\Omega_E^+$ and $\Omega_{E^K}^+$ denote the real periods of $E/\QQ$ and $E^K/\QQ$ respectively. See \cite[p.~312]{GZ}.) We then define $b_K$ to be the image of $\gamma^\ast$ under the comparison isomorphism
\begin{equation}\label{comparison}
{\bigwedge}_{\ZZ_{p}}^2 H_1(E(\CC),\ZZ_{p})^\ast \simeq {\bigwedge}_{\ZZ_p}^2 T(-1). 
\end{equation}

Let
$$\eta_K=\eta_{K/K,S,\emptyset}(T) \in \CC_p\otimes_{\ZZ_p}{\bigwedge}_{\ZZ_p}^2 H^1(\cO_{K,S},T)$$
be the special element defined by using $b_K$ (see Definition \ref{special}). 
A relation between $z_K^{\rm Hg}$ and $\eta_K$ is given as follows. 

\begin{proposition}\label{GZ}
Assume ${\ord}_{s=1}L(E/K,s)=1$ (which implies ${\rm rank}(E(K))=1$ and $\#\sha(E/K)<\infty$ by the well-known theorem of Gross-Zagier-Kolyvagin). Then we have
$$z_K^{\rm Hg}=\eta_K.$$
In particular, $\eta_K$ lies in ${\bigwedge}_{\QQ_p}^2 H^1(\cO_{K,S},V)$. 
\end{proposition}

\begin{proof}
Let 
$$\lambda_{T,K}: \CC_p\otimes_{\ZZ_p} {\bigwedge}_{\ZZ_p}^2 H^1(\cO_{K,S},T) \xrightarrow{\sim} \CC_p \otimes_{\ZZ_p} {\bigwedge}_{\ZZ_p}^2 Y_K(T^\ast(1))^\ast$$
be the period-regulator isomorphism defined in \S \ref{def per}. By definition, the special element $\eta_K$ is characterized by
$$\lambda_{T,K}(\eta_K) = L_S^\ast(E/K,1)\cdot b_K.$$
So it is sufficient to show that
$$\lambda_{T,K}(z_K^{\rm Hg}) =L_S^\ast(E/K,1)\cdot b_K.$$

Under the assumption, one checks that $\lambda_{T,K}$ coincides with the following composition map:
\begin{eqnarray*}
\CC_p \otimes_{\ZZ_p} {\bigwedge}_{\ZZ_p}^2 H^1(\cO_{K,S},T) & \stackrel{\eqref{isom2}}{\simeq}& \CC_p \otimes_\QQ \left( E(K)\otimes_\ZZ E(K) \otimes_\ZZ {\bigwedge}_\QQ^2 \Gamma(E,\Omega_{E/K}^1)\right) \\
& \simeq& \CC_p \otimes_\QQ {\bigwedge}_\QQ^2 \Gamma(E,\Omega_{E/K}^1) \\
&\stackrel{\eqref{period map}}{\simeq}& \CC_p\otimes_\QQ {\bigwedge}_\QQ^2 H_1(E(\CC),\QQ)^\ast \\
&\stackrel{\eqref{comparison}}{\simeq}& \CC_p\otimes_{\ZZ_p} {\bigwedge}_{\ZZ_p}^2 T(-1) = \CC_p\otimes_{\ZZ_p} {\bigwedge}_{\ZZ_p}^2 Y_K(T^\ast(1))^\ast,
\end{eqnarray*}
where the second isomorphism is induced by the N\'eron-Tate height pairing 
$$\langle -,-\rangle_{\infty,K}: E(K)\times E(K) \to \RR.$$
By the definition of $z_K^{\rm Hg}$, we have
$$\lambda_{T,K}(z_K^{\rm Hg}) = {\rm Eul}_S\cdot (u_Kc_\phi)^{-2} \cdot \langle y_K,y_K\rangle_{\infty,K} \cdot \frac{1}{\sqrt{|D_K|}} \cdot \Omega_{E/K}\cdot b_K.$$
By the Gross-Zagier formula \cite{GZ}, we know that
$$L'(E/K,1)= (u_Kc_\phi)^{-2} \cdot \langle y_K,y_K\rangle_{\infty,K} \cdot \frac{1}{\sqrt{|D_K|}} \cdot \Omega_{E/K},$$
so we have
$$\lambda_{T,K}(z_K^{\rm Hg})= {\rm Eul}_S \cdot L'(E/K,1)\cdot b_K = L_S^\ast(E/K,1)\cdot b_K.$$
This proves the proposition. 
\end{proof}

Let ${\rm Tam}(E/K)$ be the product of Tamagawa factors and $R_{E/K}$ the N\'eron-Tate regulator for $E/K$. Recall that the $p$-part of the Birch-Swinnerton-Dyer formula for $E/K$ predicts the equality (in $\CC_p$)
$$\ZZ_p\cdot L^\ast(E/K,1)=\ZZ_p\cdot \frac{\# \sha(E/K)\cdot {\rm Tam}(E/K)}{\# E(K)_{\rm tors}^2}\cdot \frac{1}{\sqrt{|D_K|}} \Omega_{E/K} \cdot R_{E/K}.$$

\begin{proposition}\label{prop bsd}
Assume ${\ord}_{s=1}L(E/K,s)=1$ and $E(K)[p]=0$. Then the $p$-part of the Birch-Swinnerton-Dyer formula for $E/K$ holds if and only if we have an equality of $\ZZ_p$-modules
$$\ZZ_p \cdot z_K^{\rm Hg}= \# H^2(\cO_{K,S},T) \cdot {\bigwedge}_{\ZZ_p}^2 H^1(\cO_{K,S},T). $$
In particular, the $p$-part of the Birch-Swinnerton-Dyer formula for $E/K$ implies the ``integrality" of $z_K^{\rm Hg}$:
$$z_K^{\rm Hg}\in {\bigwedge}_{\ZZ_p}^2 H^1(\cO_{K,S},T).$$
\end{proposition}

\begin{proof}
Note that the assumption $E(K)[p]=0$ implies that $H^1(\cO_{K,S},T)$ is $\ZZ_p$-free. Since the $p$-part of the Birch-Swinnderton-Dyer formula for $E/K$ is equivalent to the Tamagawa number conjecture for $h^1(E/K)(1)$ (with coefficients in $\ZZ_p$), the proposition follows immediately from Propositions \ref{tnc ext} and \ref{GZ}. 
\end{proof}

The following result gives a connection between Heegner points and rank two Euler systems. 

\begin{theorem}\label{heeg bottom}
Assume ${\ord}_{s=1}L(E/K,s)=1$, $E(K)[p]=0$ and the $p$-part of the Birch-Swinnerton-Dyer formula for $E/K$ holds. Then for any abelian $p$-extension $\cK/K$ there exists a rank two Euler system $c \in {\rm ES}_2(T,\cK)$ such that
$$c_K=z_K^{\rm Hg}.$$
\end{theorem}

\begin{proof}
Note that we have a canonical isomorphism
$$\Theta_{T,K}:{\det}_{\ZZ_p}^{-1}(\rgamma(\cO_{K,S},T)) \simeq \# H^2(\cO_{K,S},T) \cdot {\bigwedge}_{\ZZ_p}^2 H^1(\cO_{K,S},T).$$
Since we assume the $p$-part of the Birch-Swinnerton-Dyer formula, Proposition \ref{prop bsd} implies that there is a basis $\fz_K^{\rm Hg} \in {\det}_{\ZZ_p}^{-1}(\rgamma(\cO_{K,S},T))$ such that $\Theta_{T,K}(\fz_K^{\rm Hg})= z_K^{\rm Hg}$. Let ${\rm VS}(T,\cK)$ be the module mentioned in Remark \ref{remark sbA}, which is defined to be a certain inverse limit $\varprojlim_{F\in \Omega(\cK)} {\det}_{\ZZ_p[\cG_F]}^{-1}(\rgamma(\cO_{F,S(F)},T))$ with surjective transition maps. Then we have a commutative diagram
$$
\xymatrix{
{\rm VS}(T,\cK) \ar[r]^-{\Theta_{T,\cK}} \ar@{->>}[d] & {\rm ES}_2(T,\cK) \ar[d]\\
{\det}_{\ZZ_p}^{-1}(\rgamma(\cO_{K,S},T)) \ar[r]_-{\Theta_{T,K}} &{\bigwedge}_{\ZZ_p}^2 H^1(\cO_{K,S},T),
}
$$
where the left vertical surjection is the natural projection map and the right vertical arrow sends $c$ to $c_K$. Take a lift
$$\fz \in {\rm VS}(T,\cK)$$
of $\fz_K^{\rm Hg} \in {\det}_{\ZZ_p}^{-1}(\rgamma(\cO_{K,S},T))$ and put
$$c:=\Theta_{T,\cK}(\fz) \in {\rm ES}_2(T,\cK).$$
Then this Euler system has the desired property. 
\end{proof}

\begin{remark}
The Euler system constructed in Theorem \ref{heeg bottom} is not canonical. A canonical rank two Euler system should be constructed {\it directly} from Heegner points over ring class fields (so that it satisfies the properties (i) and (ii) in Conjecture \ref{conjes}). 
\end{remark}

\subsection{The Heegner point main conjecture}

In this subsection, we relate our formulation of the Iwasawa main conjecture (Conjecture \ref{neIMC}) with the Heegner point main conjecture formulated by Perrin-Riou \cite{PRheeg} (which has been studied in many works including \cite{bertolini}, \cite{howard}, \cite{wan}, \cite{castella} and \cite{BCK}). 

In the following, {\it we assume $E$ has good ordinary reduction at $p$}. 

\subsubsection{Formulation of the Heegner point main conjecture}
We first review the formulation of the Heegner point main conjecture. 
Let $K_\infty/K$ be the anticyclotomic $\ZZ_p$-extension and $K_n$ its $n$-th layer. We set
$$\Gamma_n:=\Gal(K_n/K), \ \Gamma:=\Gal(K_\infty/K) \text{ and }\Lambda:=\ZZ_p[[\Gamma]]\simeq \varprojlim_n \ZZ_p[\Gamma_n].$$
We also set
$$\TT:=\varprojlim_n {\rm Ind}_{K_n/K}(T), \ W:=V/T \text{ and }\WW:=\varinjlim_n {\rm Ind}_{K_n/K}(W).$$

For $X \in \{\TT,\WW\}$, we define a $\Lambda$-adic Selmer group ${\rm Sel}(X)$ as follows. Note first that, since $E$ has good ordinary reduction at each $ v\in S_p(K)$, we have a natural filtration $F^+ X \subset X$ (as $G_{K_v}$-modules). We set $F^-X:=X/F^+X$ and define
$${\rm Sel}(X):= \ker \left(H^1(\cO_{K,S},X) \to \bigoplus_{v \in S_p(K)} H^1(K_v,F^-X) \oplus \bigoplus_{v \in S\setminus (S_\infty(K)\cup S_p(K))} H^1(K_v^{\rm ur},X) \right),$$
where $K_v^{\rm ur}$ denotes the maximal unramified extension of $K_v$. 

We set
$${\rm Sel}(\WW)^\vee:=\Hom_{\ZZ_p}({\rm Sel}(\WW),\QQ_p/\ZZ_p) \text{ and }\sha_\infty:= ({\rm Sel}(\WW)^\vee)_{\rm tors}.$$

Let
$$y_\infty \in {\rm Sel}(\TT)$$
be the $\Lambda$-adic Heegner class, which is denoted 
by ${\bf z}_f$ in \cite[\S 3.1]{castella}. By \cite{cornut} and \cite{CV}, we know that $y_\infty$ is non-torsion. 

Lastly, let $ \iota: \Lambda \to \Lambda; \ a \mapsto a^\iota$ denote the involution induced by $\Gamma \to \Gamma; \ \gamma\mapsto \gamma^{-1}$. 


\begin{conjecture}[The Heegner point main conjecture] \label{hpmc}
We have
$${\rm char}_\Lambda({\rm Sel}(\TT)/\Lambda\cdot y_\infty)\cdot  {\rm char}_\Lambda({\rm Sel}(\TT)/\Lambda\cdot y_\infty)^\iota= {\rm char}_\Lambda(\sha_\infty).$$
\end{conjecture}

\begin{remark}
Building on works by many people (including Howard \cite{howardbi} and Zhang \cite{zhang}), Burungale-Castella-Kim has recently proved Conjecture \ref{hpmc} under mild hypotheses (see \cite[Th.~A]{BCK}). 
\end{remark}

\subsubsection{Selmer complexes}

For later purpose, we give another formulation of the Heegner point main conjecture by using Nekov\'a\v r's Selmer complexes \cite{nekovar}. 

Let
$$\widetilde \rgamma_f(\TT):=\widetilde \rgamma_{f, {\rm Iw}}(K_\infty/K,T) \text{ and }\widetilde \rgamma_f(\WW):=\widetilde \rgamma_f(K_S/K_\infty, W)$$
be Selmer complexes defined in \cite[(8.8.5)]{nekovar}. We write $\widetilde H^i_f(-)$ for $H^i(\widetilde \rgamma_f(-))$. 

\begin{lemma} \label{lemselmer}\
\begin{itemize}
\item[(i)] Let $X \in \{\TT, \WW\}$. There is a canonical exact triangle
$$\widetilde \rgamma_f(X) \to \rgamma(\cO_{K,S},X) \to \bigoplus_{v\in S_p(K)}\rgamma(K_v,F^-X)\oplus \bigoplus_{v\in S\setminus (\{\infty\}\cup S_p(K))}U_v^-(X),$$
where $U_v^-(X)$ is defined by
$$U_v^-(X):={\rm Cone}\left(\rgamma(K_v^{\rm ur}/K_v,X^{G_{K_v^{\rm ur}}}) \to \rgamma(K_v,X)\right).$$
\item[(ii)] We have
$$\widetilde H^1_f(\TT)={\rm Sel}(\TT).$$
\item[(iii)] There is a canonical surjection
$$\widetilde H^1_f(\WW)\twoheadrightarrow {\rm Sel}(\WW)$$
with finite kernel. 
\item[(iv)] There is a canonical isomorphism
$$\widetilde H^i_f(\TT) \simeq (\widetilde H^{3-i}_f(\WW)^\iota)^\vee ,$$
where we write $(-)^\iota$ for the module on which $\Lambda$ acts via the involution $\iota$ and $(-)^\vee$ for the Pontryagin dual. 
\item[(v)] There is a canonical map
$$\widetilde H^2_f(\TT) \to \Hom_\Lambda(\widetilde H^1_f(\TT),\Lambda)^\iota$$
such that 
\begin{itemize}
\item the cokernel is pseudo-null,
\item the kernel is pseudo-isomorphic to $\sha_\infty^\iota$.
\end{itemize}
\end{itemize}
\end{lemma}
\begin{proof}
Claim (i) follows from \cite[(6.1.3.2)]{nekovar}. Claim (ii) follows from claim (i) by noting that $H^0(K_v,F^-\TT)=0$ for $v\in S_p(K)$ and $H^0(U_v^-(\TT))=0$ for $v \notin \{\infty\}\cup S_p(K)$. 
Claim (iii) again follows from claim (i) by noting that $H^0(K_v,F^-\WW)$ is finite for $v \in S_p(K)$ and $H^0(U_v^-(\WW))=0$ for $v \notin \{\infty\}\cup S_p(K)$. Claim (iv) follows from the duality \cite[(8.9.6.2)]{nekovar} and the natural identification $W=T^\vee(1)$ induced by the Weil pairing. 

We prove claim (v). By \cite[Th.~8.9.9]{nekovar}, we have a short exact sequence
$$0\to \widetilde H^2_f(\TT)_{\rm tors}\to \widetilde H^2_f(\TT)\to \Hom_\Lambda(\widetilde H^1_f(\TT), \Lambda)^\iota \to 0$$
modulo pseudo-null. By (iii) and (iv), we see that $\widetilde H^2_f(\TT)_{\rm tors}$ is pseudo-isomorphic to $\sha_\infty^\iota$. This proves the claim. 
\end{proof}

\begin{remark}
Combining Lemma \ref{lemselmer}(ii), (iii), (iv) and (v), we obtain a canonical isomorphism
$$Q(\Lambda)\otimes_\Lambda {\rm Sel}(\TT) \simeq \Hom_\Lambda({\rm Sel}(\WW)^\vee, Q(\Lambda)),$$
where $Q(\Lambda)$ denotes the quotient field of $\Lambda$.
In particular, the $\Lambda$-ranks of ${\rm Sel}(\TT)$ and ${\rm Sel}(\WW)^\vee$ are the same. 
\end{remark}

We now assume the following mild hypothesis. 

\begin{hypothesis}\label{hyp1}
The $\Lambda$-module ${\rm Sel}(\TT)$ is torsion-free and of rank one (i.e., ${\dim}_{Q(\Lambda)}(Q(\Lambda)\otimes_\Lambda {\rm Sel}(\TT))=1$).
\end{hypothesis}

\begin{remark}
Hypothesis \ref{hyp1} is proved in \cite[Th.~B]{howard} under mild assumptions (see also \cite[Th.~A]{bertolini} and \cite[\S 2]{nekovar parity}). 
\end{remark}

Under this hypothesis, we define a canonical isomorphism
\begin{equation}\label{selmer isom}
Q(\Lambda)\otimes_\Lambda{\det}_\Lambda^{-1}(\widetilde \rgamma_f (\TT)) \simeq Q(\Lambda)\otimes_\Lambda {\rm Sel}(\TT)\otimes_\Lambda {\rm Sel}(\TT)^\iota
\end{equation}
by the composition
\begin{eqnarray*}
Q(\Lambda)\otimes_\Lambda{\det}_\Lambda^{-1}(\widetilde \rgamma_f (\TT)) &\simeq& Q(\Lambda) \otimes_\Lambda {\det}_\Lambda(\widetilde H^1_f(\TT)) \otimes_\Lambda {\det}_\Lambda^{-1}(\widetilde H^2_f(\TT)) \\
&\simeq& Q(\Lambda) \otimes_\Lambda {\det}_\Lambda({\rm Sel}(\TT)) \otimes_\Lambda {\det}_\Lambda({\rm Sel}(\TT)^\iota) \\
&\simeq &Q(\Lambda)\otimes_\Lambda {\rm Sel}(\TT)\otimes_\Lambda {\rm Sel}(\TT)^\iota,
\end{eqnarray*}
where the first isomorphism follows by noting that $\widetilde H^0_f(\TT)=0$ and $\widetilde H^3_f(\TT)=( \widetilde H^0_f(\WW)^\iota)^\vee$ is finite (which follows from the fact that $E(K_\infty)[p^\infty]$ is finite, see \cite[Lem.~2.1(v)]{nekovar parity}), the second by Lemma \ref{lemselmer}(ii) and (v), and the last by Hypothesis \ref{hyp1}. 

\begin{proposition}\label{selmer formulation}
Assume Hypothesis \ref{hyp1}. Then Conjecture \ref{hpmc} holds if and only if there is a $\Lambda$-basis
$$\widetilde \fz_\infty \in {\det}_\Lambda^{-1}(\widetilde \rgamma_f(\TT))$$
such that the map \eqref{selmer isom} sends $\widetilde \fz_\infty$ to $y_\infty \otimes y_\infty$. 
\end{proposition}
\begin{proof}
By Lemma \ref{lemselmer}(ii) and (v), we have a canonical isomorphism
$${\det}_\Lambda^{-1}(\widetilde H^2_f(\TT)) \simeq {\det}_\Lambda^{-1}(\sha_\infty^\iota)\otimes_\Lambda {\det}_\Lambda({\rm Sel}(\TT)^\iota).$$
Since we have ${\det}_\Lambda^{-1}(\sha_\infty^\iota)\simeq {\rm char}_\Lambda(\sha_\infty)^\iota$ (see the argument in the proof of Lemma \ref{alg lemma}(ii)), we see that the image of ${\det}_\Lambda^{-1}(\widetilde \rgamma_f(\TT))$ under the map \eqref{selmer isom} coincides with 
$${\rm char}_\Lambda(\sha_\infty)^\iota\cdot {\rm Sel}(\TT)\otimes_\Lambda {\rm Sel}(\TT)^\iota.$$
The claim follows easily from this. 
\end{proof}

\subsubsection{$\Lambda$-adic Heegner elements}

To simplify the notation, we set
$$\HH^i:=H^i(\cO_{K,S},\TT).$$
We shall define a ``$\Lambda$-adic Heegner element"
$$z_\infty^{\rm Hg} \in Q(\Lambda)\otimes_\Lambda {\bigcap}_\Lambda^2 \HH^1$$
and compare the Iwasawa main conjecture (Conjecture \ref{neIMC}) for $z_{\infty}^{\rm Hg}$ with the Heegner point main conjecture. 

We fix an isomorphism
\begin{equation}\label{anti2}
{\det}_\Lambda\left( \bigoplus_{v\in S_p(K)}\rgamma(K_v,F^-\TT)\oplus \bigoplus_{v\in S\setminus (\{\infty\}\cup S_p(K))}U_v^-(\TT)\right)\simeq \Lambda.
\end{equation}
Then the exact triangle in Lemma \ref{lemselmer}(i) induces an isomorphism
\begin{equation}\label{selmers}
 {\det}_\Lambda^{-1}(\widetilde \rgamma_f(\TT))  \simeq {\det}_\Lambda^{-1}(\rgamma(\cO_{K,S},\TT)) .
\end{equation}

\begin{definition}\label{def lambda}
Assume Hypotheses \ref{leop} (i.e., $\HH^2$ is $\Lambda$-torsion) and \ref{hyp1}. We define the {\it $\Lambda$-adic Heegner element}
$$z_\infty^{\rm Hg} \in Q(\Lambda)\otimes_\Lambda {\bigcap}_\Lambda^2 \HH^1$$
as the image of 
$$y_\infty \otimes y_\infty \in {\rm Sel}(\TT) \otimes_\Lambda {\rm Sel}(\TT)^\iota $$
under the map
\begin{eqnarray*}
{\rm Sel}(\TT) \otimes_\Lambda {\rm Sel}(\TT)^\iota &\stackrel{\eqref{selmer isom}}{\rightarrow} &Q(\Lambda)\otimes_\Lambda {\det}_\Lambda^{-1}(\widetilde \rgamma_f(\TT)) \\&\stackrel{\eqref{selmers}}{\simeq} &Q(\Lambda)\otimes_\Lambda {\det}_\Lambda^{-1}(\rgamma(\cO_{K,S},\TT)) \\
&\stackrel{\eqref{canisom2}}{\simeq}& Q(\Lambda)\otimes_\Lambda {\bigcap}_\Lambda^2 \HH^1.
\end{eqnarray*}
(Note that \eqref{canisom2} and \eqref{selmer isom} are defined under Hypotheses \ref{leop} and \ref{hyp1} respectively.)
\end{definition}

\begin{remark}
The element $z_\infty^{\rm Hg}$ is canonical up to $\Lambda^\times$ (it depends on the choice of an isomorphism \eqref{anti2}). 
\end{remark}

\begin{remark}\label{anti leop}
Note that Hypothesis \ref{leop} is equivalent to
$$H^2(\cO_{K,S},\WW)=\varinjlim_n H^2(\cO_{K_n,S},E[p^\infty])=0.$$
(See \cite[Prop.~1.3.2]{PRast} or \cite[Lem.~9.1.5]{nekovar}.) The vanishing of $H^2(\cO_{K,S},\WW)$ is proved by Bertolini \cite[Th.~5.4]{bertolinileop} under mild assumptions. Thus Hypothesis \ref{leop} is known to hold under mild assumptions. 
\end{remark}

The following is our formulation of the non-equivariant Iwasawa main conjecture (Conjecture \ref{neIMC}) in the present setting. 

\begin{conjecture}[{The Iwasawa main conjecture for $z_\infty^{\rm Hg}$}]\label{hpmc2}
Assume Hypotheses \ref{leop} and \ref{hyp1}. Then we have $z_\infty^{\rm Hg} \in {\bigcap}_\Lambda^2\HH^1$ and 
$${\rm char}_\Lambda\left({\bigcap}_\Lambda^2 \HH^1/\Lambda\cdot z_\infty^{\rm Hg}\right) ={\rm char}_\Lambda(\HH^2).$$
\end{conjecture}



\begin{theorem}\label{heeg equivalent}
Assume Hypotheses \ref{leop} and \ref{hyp1}. Then the Heegner point main conjecture (Conjecture \ref{hpmc}) is equivalent to Conjecture \ref{hpmc2}. 
\end{theorem}

\begin{proof}
This follows immediately from Propositions \ref{selmer formulation} and \ref{imc equivalent}. 
\end{proof}

The following is an Iwasawa theoretic analogue of Theorem \ref{heeg bottom}. 

\begin{theorem}\label{heeg iwasawa}
Assume Hypotheses \ref{leop} and \ref{hyp1}, and also the Heegner point main conjecture (Conjecture \ref{hpmc}). Then, for any abelian $p$-extension $\cK/K$ containing $K_\infty$, there exists a rank two Euler system $c\in {\rm ES}_2(T,\cK)$ such that
$$\varprojlim_n c_{K_n} =z_\infty^{\rm Hg} \text{ in }\varprojlim_n {\bigcap}_{\ZZ_p[\Gamma_n]}^2 H^1(\cO_{K_n,S},T) \simeq {\bigcap}_\Lambda^2 \HH^1. $$
\end{theorem}

\begin{proof}
The assumed validity of the Heegner point main conjecture implies the existence of a $\Lambda$-basis 
$$\fz_\infty^{\rm Hg} \in {\det}_\Lambda^{-1}(\rgamma(\cO_{K,S},\TT))$$
such that the map
$$\Theta_{T,K_\infty}: {\det}_\Lambda^{-1}(\rgamma(\cO_{K,S},\TT)) \to {\bigcap}_\Lambda^2 \HH^1$$
induced by \eqref{canisom2} sends $\fz_\infty^{\rm Hg}$ to $z_\infty^{\rm Hg}$. Similarly to the proof of Theorem \ref{heeg bottom}, we have a commutative diagram
$$
\xymatrix{
{\rm VS}(T,\cK) \ar[r]^-{\Theta_{T,\cK}} \ar@{->>}[d] & {\rm ES}_2(T,\cK) \ar[d]\\
{\det}_{\Lambda}^{-1}(\rgamma(\cO_{K,S},\TT)) \ar[r]_-{\Theta_{T,K_\infty}} &{\bigcap}_{\Lambda}^2 \HH^1,
}
$$
where the right vertical arrow sends $c$ to $\varprojlim_n c_{K_n}$. Take a lift
$$\fz \in {\rm VS}(T,\cK)$$
of $\fz_\infty^{\rm Hg} \in {\det}_{\Lambda}^{-1}(\rgamma(\cO_{K,S},\TT))$ and put
$$c:=\Theta_{T,\cK}(\fz) \in {\rm ES}_2(T,\cK).$$
Then this Euler system has the desired property. 
\end{proof}

\begin{remark}
We expect that the element $z_\infty^{\rm Hg} $ coincides with $\varprojlim_n \eta_{K_n}$ up to normalization, where 
$$\eta_{K_n}=\eta_{K_n/K,S,\emptyset}(T)\in \CC_p \otimes_{\ZZ_p} {\bigwedge}_{\ZZ_p[\Gamma_n]}^2 H^1(\cO_{K_n,S},T)$$
is the special element in Definition \ref{special}. Since $\varprojlim_n \eta_{K_n}$ is defined even in the supersingular and bad reduction cases, we expect there is a construction of $z_\infty^{\rm Hg} $ without assuming $E$ has good ordinary reduction at $p$. 
\end{remark}

\subsection{Derivatives of Heegner elements}

In this subsection, we study Conjecture \ref{der iw} for the $\Lambda$-adic Heegner element $z_\infty^{\rm Hg}$ in Definition \ref{def lambda}. We keep assuming that $E$ has good ordinary reduction at $p$. {\it In this subsection, we always assume Hypotheses \ref{leop} (i.e., $\HH^2$ is $\Lambda$-torsion) and \ref{hyp1}.} 

We set $e:=\dim_{\QQ_p}(H^2(\cO_{K,S},V))$. We also set
$$ I_n:=\ker (\ZZ_p[\Gamma_n]\twoheadrightarrow \ZZ_p) \text{ and }I:=\ker (\Lambda \twoheadrightarrow \ZZ_p)\simeq \varprojlim_n I_n.$$ 
Let
\begin{eqnarray*}
\iota_n:=\iota_{K_n}: {\bigwedge}_{\ZZ_p}^2 H^1(\cO_{K,S},T) \otimes_{\ZZ_p} I_n^e/I_n^{e+1} &\hookrightarrow& {\bigcap}_{\ZZ_p[\Gamma_n]}^2 H^1(\cO_{K_n,S},T)\otimes_{\ZZ_p} I_n^e/I_n^{e+1} \\
&\subset& {\bigcap}_{\ZZ_p[\Gamma_n]}^2 H^1(\cO_{K_n,S},T)\otimes_{\ZZ_p} \ZZ_p[\Gamma_n]/I_n^{e+1} 
\end{eqnarray*}
be the canonical injection in \S \ref{sec der finite}. 

\begin{proposition}
Assume the Heegner point main conjecture (Conjecture \ref{hpmc}). Then there exists the Darmon derivative of $z_\infty^{\rm Hg}$:
$$\kappa_\infty^{\rm Hg}=\varprojlim_n \kappa_n^{\rm Hg} \in {\bigwedge}_{\ZZ_p}^2 H^1(\cO_{K,S},T)\otimes_{\ZZ_p} \varprojlim_n I_n^e/I_n^{e+1}={\bigwedge}_{\ZZ_p}^2 H^1(\cO_{K,S},T)\otimes_{\ZZ_p}  I^e/I^{e+1},$$
i.e., the unique element satisfying
$$\iota_{n}(\kappa_n^{\rm Hg})=\sum_{\sigma \in \Gamma_n} \sigma z_n^{\rm Hg} \otimes \sigma^{-1} \text{ in }{\bigcap}_{\ZZ_p[\Gamma_n]}^2 H^1(\cO_{K_n,S},T)\otimes_{\ZZ_p}\ZZ_p[\Gamma_n]/I_n^{e+1}$$
for every $n$. (Note that Theorem \ref{heeg equivalent} implies that $z_\infty^{\rm Hg}$ lies in ${\bigcap}_\Lambda^2 \HH^1 = \varprojlim_n {\bigcap}_{\ZZ_p[\Gamma_n]}^2 H^1(\cO_{K_n,S},T)$, and $z_n ^{\rm Hg}\in {\bigcap}_{\ZZ_p[\Gamma_n]}^2 H^1(\cO_{K_n,S},T)$ denotes the element such that $z_\infty^{\rm Hg}=\varprojlim_n z_n^{\rm Hg}$.)
\end{proposition}

\begin{proof}
This follows from Theorem \ref{heeg equivalent} and Proposition \ref{prop dar}. 
\end{proof}

\begin{remark}
Since $z_\infty^{\rm Hg}$ is canonical up to $\Lambda^\times$, its Darmon derivative $\kappa_\infty^{\rm Hg}$ is canonical up to $\ZZ_p^\times$. 
\end{remark}

In the following, we assume the following hypothesis. (Recall that $E^K/\QQ$ denotes the quadratic twist of $E/\QQ$ by $K$.)

\begin{hypothesis}\label{hypk}\
\begin{itemize}
\item[(i)] $E(K)[p]=0$;
\item[(ii)] $r^+:={\rm rank}(E(\QQ))>0$ and $r^-:={\rm rank}(E^K(\QQ))>0$ (in particular, ${\rm rank}(E(K))\geq 2$);
\item[(iii)] $\#\sha(E/K)[p^\infty]<\infty$.
\end{itemize}
\end{hypothesis}

\begin{lemma}\label{leme}
Assume Hypothesis \ref{hypk}.
\begin{itemize}
\item[(i)] We have canonical isomorphisms
$$H^1(\cO_{K,S},V)\simeq\QQ_p\otimes_\ZZ E(K) \simeq \QQ_p\otimes_\ZZ(E(\QQ)\oplus E^K(\QQ)) .$$
\item[(ii)] We have a canonical exact sequence
\begin{equation}\label{gamma exact}
0\to \QQ_p\otimes_{\QQ}\Gamma(E,\Omega_{E/K}^1) \to \QQ_p\otimes_\ZZ E(K)^\ast \to H^2(\cO_{K,S},V)\to 0.
\end{equation}
In particular, we have
$$e:=\dim_{\QQ_p}(H^2(\cO_{K,S},V))= r^+ + r^- -2.$$
\end{itemize}
\end{lemma}

\begin{proof}
By \eqref{long exact}, it is sufficient to prove that the map $\QQ_p\otimes_\QQ \Gamma(E,\Omega_{E/K}^1) \to \QQ_p\otimes_\ZZ E(K)^\ast$ is injective, or equivalently, the localization map $\QQ_p\otimes_\ZZ E(K) \to \QQ_p \otimes_{\ZZ_p} (\varprojlim_n E(K_p)/p^n)$ is surjective. Since we have 
$$\QQ_p\otimes_\ZZ E(K) = \QQ_p\otimes_\ZZ (E(\QQ)\oplus E^K(\QQ)) $$
and 
$$\QQ_p \otimes_{\ZZ_p} (\varprojlim_n E(K_p)/p^n) = \QQ_p \otimes_{\ZZ_p} (\varprojlim_n E(\QQ_p)/p^n \oplus \varprojlim_n E^K(\QQ_p)/p^n),$$
it is sufficient to prove the surjectivity of $\QQ_p \otimes_\ZZ A(\QQ) \to \QQ_p\otimes_{\ZZ_p} (\varprojlim_n A(\QQ_p)/p^n)$ for $A \in \{E, E^K\}$. However, this is true by Hypothesis \ref{hypk}(ii). 
\end{proof}

\begin{remark}
The Heegner hypothesis implies that $\ord_{s=1}L(E/K,s)$ is odd, and so by the parity conjecture \cite{nekovar parity} we know that ${\rm rank}(E(K))$ is also odd. So by Lemma \ref{leme}(ii) we have $e>0$ under Hypothesis \ref{hypk}. 
\end{remark}

We shall define a canonical ``anticyclotomic Bockstein regulator"
$$R_{K_\infty}^{\rm Boc} \in {\bigwedge}_{\QQ_p}^2 H^1(\cO_{K,S},V) \otimes_{\ZZ_p}I^e/I^{e+1}$$
as follows. 

We fix a $\ZZ_p$-basis $x \in {\bigwedge}_{\ZZ_p}^e H^2(\cO_{K,S},T)_{\rm tf}$, and let
$${\rm Boc}_{\infty}=\varprojlim_n {\rm Boc}_{T,K_n,x}: {\bigwedge}_{\QQ_p}^{e+2} H^1(\cO_{K,S},V) \to {\bigwedge}_{\QQ_p}^2 H^1(\cO_{K,S},V) \otimes_{\ZZ_p}I^e/I^{e+1}$$
be the limit of Bockstein regulator maps defined in \S \ref{sec boc}. By Lemma \ref{leme}, we have ${\rm rank}(E(K))=r^++r^-=e+2$ and we fix a $\ZZ$-basis $\{P_1,\ldots,P_{e+2}\}$ of $E(K)_{\rm tf}$, which is regarded as a $\QQ_p$-basis of $H^1(\cO_{K,S},V)$ by identifying $\QQ_p\otimes_\ZZ E(K)=H^1(\cO_{K,S},V)$. Let $\{P_1^\ast,\ldots,P_{e+2}^\ast\}$ be the basis of $H^1(\cO_{K,S},V)^\ast$ which is dual to $\{P_1,\ldots,P_{e+2}\}$. 

Recall that we fixed N\'eron differentials $\omega$ and $\omega^K$ of $E/\QQ$ and $E^K/\QQ$ respectively. Then $\{\omega,\omega^K\}$ is a $\QQ$-basis of $\Gamma(E,\Omega_{E/K}^1)$ and so we can identify $\Gamma(E,\Omega_{E/K}^1)=\QQ^2$. By the exact sequence \eqref{gamma exact}, we get an isomorphism
\begin{equation}\label{omega isom}
{\bigwedge}_{\QQ_p}^eH^2(\cO_{K,S},V) \simeq {\bigwedge}_{\QQ_p}^{e+2} (\QQ_p\otimes_\ZZ E(K))^\ast ={\bigwedge}_{\QQ_p}^{e+2} H^1(\cO_{K,S},V)^\ast.
\end{equation}

\begin{definition}
Assume Hypothesis \ref{hypk}. We define the {\it anticyclotomic Bockstein regulator} by
$$R_{K_\infty}^{\rm Boc}:=C_x \cdot {\rm Boc}_\infty(P_1\wedge \cdots \wedge P_{e+2})\in {\bigwedge}_{\QQ_p}^2 H^1(\cO_{K,S},V) \otimes_{\ZZ_p}I^e/I^{e+1},$$
where $C_x \in \QQ_p^\times$ is the element satisfying 
$$x \stackrel{\eqref{omega isom}}{\mapsto} C_x \cdot P_1^\ast \wedge \cdots \wedge P_{e+2}^\ast.$$
(One checks that $R_{K_\infty}^{\rm Boc}$ is independent of the choice of $x$ and $\{P_1,\ldots,P_{e+2}\}$.)
\end{definition}

Recall that Conjecture \ref{der iw} for $z_\infty^{\rm Hg}$ predicts the equality
$$\kappa_\infty^{\rm Hg}={\rm Boc}_\infty (\widetilde \eta_K) \text{ in }\CC_p\otimes_{\ZZ_p}{\bigwedge}_{\ZZ_p}^2 H^1(\cO_{K,S},T)\otimes_{\ZZ_p}I^e/I^{e+1},$$
where $\widetilde \eta_K=\widetilde \eta_{K,S,\emptyset}(T) \in \CC_p\otimes_{\ZZ_p}{\bigwedge}_{\ZZ_p}^{e+2} H^1(\cO_{K,S},T)$ is the extended special element in Definition \ref{def ext} (with respect to $b_K \in {\bigwedge}_{\ZZ_p}^2 Y_K(T^\ast(1))^\ast$ in \eqref{bk} and $x \in {\bigwedge}_{\ZZ_p}^e H^2(\cO_{K,S},T)_{\rm tf}$ fixed above). 

\begin{proposition}\label{explicit derivative}
Assume Hypothesis \ref{hypk}. Then Conjecture \ref{der iw} for $z_\infty^{\rm Hg}$ holds if and only if we have an equality
$$\kappa_\infty^{\rm Hg}= \frac{L_S^\ast(E/K,1)\sqrt{|D_K|}}{\Omega_{E/K} \cdot R_{E/K}}\cdot R_{K_\infty}^{\rm Boc}\text{ in }\CC_p\otimes_{\ZZ_p}{\bigwedge}_{\ZZ_p}^2 H^1(\cO_{K,S},T)\otimes_{\ZZ_p}I^e/I^{e+1}.$$
\end{proposition}

\begin{proof}
It is sufficient to prove that
$$\widetilde \eta_K = \frac{L_S^\ast(E/K,1)\sqrt{|D_K|}}{\Omega_{E/K} \cdot R_{E/K}} \cdot C_x \cdot P_1\wedge \cdots \wedge P_{e+2}.$$
Let 
$$\widetilde \lambda_{T,K}: \CC_p\otimes_{\ZZ_p} {\bigwedge}_{\ZZ_p}^{e+2} H^1(\cO_{K,S},T) \xrightarrow{\sim} \CC_p\otimes_{\ZZ_p} \left( {\bigwedge}_{\ZZ_p}^e H^2(\cO_{K,S},T)_{\rm tf} \otimes_{\ZZ_p} {\bigwedge}_{\ZZ_p}^2 Y_K(T^\ast(1))^\ast \right)$$
be the extended period-regulator isomorphism defined in \S \ref{sec ext}. Since $\widetilde \eta_K$ is characterized by $\widetilde \lambda_{T,K}(\widetilde \eta_K) = L_S^\ast(E/K,1)\cdot (x\otimes b_K)$, it is sufficient to prove that
$$\widetilde \lambda_{T,K} (C_x \cdot P_1\wedge \cdots \wedge P_{e+2}) = \frac{\Omega_{E/K}\cdot R_{E/K}}{\sqrt{|D_K|}} \cdot (x\otimes b_K).$$
One checks that $\widetilde \lambda_{T,K}$ is explicitly given by the following composition map:
\begin{eqnarray*}
\CC_p\otimes_{\ZZ_p} {\bigwedge}_{\ZZ_p}^{e+2} H^1(\cO_{K,S},T) &=& \CC_p\otimes_\ZZ {\bigwedge}_\ZZ^{e+2} E(K) \\
&\simeq& \CC_p \otimes_\ZZ {\bigwedge}_\ZZ^{e+2} E(K)^\ast \\
&\stackrel{\eqref{gamma exact}}{\simeq}& \CC_p \otimes_{\QQ_p} \left( {\bigwedge}_{\QQ_p}^e H^2(\cO_{K,S},V) \otimes_\QQ {\bigwedge}_\QQ^2 \Gamma(E,\Omega_{E/K}^1)\right) \\
&\stackrel{\eqref{period map}}{\simeq}& \CC_p\otimes_{\QQ_p} \left({\bigwedge}_{\QQ_p}^e H^2(\cO_{K,S},V) \otimes_\QQ {\bigwedge}_\QQ^2 H_1(E(\CC),\QQ)^\ast \right) \\
&\stackrel{\eqref{comparison}}{\simeq}& \CC_p\otimes_{\ZZ_p} \left( {\bigwedge}_{\ZZ_p}^e H^2(\cO_{K,S},T)_{\rm tf} \otimes_{\ZZ_p} {\bigwedge}_{\ZZ_p}^2 Y_K(T^\ast(1))^\ast\right),
\end{eqnarray*}
where the first isomorphism is induced by the N\'eron-Tate height pairing. The claim follows easily from this description. 
\end{proof}


The following result is an analogue of \cite[Th.~7.3]{bks4}. (Recall that ${\rm Eul}_S$ is the product of Euler factors at $v \in S \setminus \{\infty\}$ satisfying ${\rm Eul}_S \cdot L^\ast(E/K,1)=L^\ast_S(E/K,1)$.)

\begin{theorem}\label{acbsd}
Assume Hypothesis \ref{hypk} and the Heegner point main conjecture (Conjecture \ref{hpmc}). Then we have
$$\ZZ_p\cdot \kappa_\infty^{\rm Hg}= \ZZ_p\cdot {\rm Eul}_S\cdot \# \sha(E/K)[p^\infty]\cdot {\rm Tam}(E/K)\cdot R_{K_\infty}^{\rm Boc} \text{ in } {\bigwedge}_{\ZZ_p}^2H^1(\cO_{K,S},T) \otimes_{\ZZ_p}I^e/I^{e+1}.$$
\end{theorem}

\begin{proof}
The Heegner point main conjecture implies the existence of a $\Lambda$-basis $\fz_\infty^{\rm Hg}$ such that the map \eqref{canisom2} sends $\fz_\infty^{\rm Hg}$ to $z_\infty^{\rm Hg}$. Let
$$\fz_K^{\rm Hg} \in {\det}_{\ZZ_p}^{-1}(\rgamma(\cO_{K,S},T))$$
be the image of $\fz_\infty^{\rm Hg}$, which is a $\ZZ_p$-basis. By the commutative diagram \eqref{diag} (for $F=K_n$), we have
$$\kappa_\infty^{\rm Hg} = {\rm Boc}_\infty(\Theta_x(\fz_K^{\rm Hg})).$$
So it is sufficient to prove that
$$\ZZ_p\cdot \Theta_x(\fz_K^{\rm Hg}) = \ZZ_p \cdot {\rm Eul}_S\cdot \# \sha(E/K)[p^\infty]\cdot {\rm Tam}(E/K) \cdot C_x \cdot P_1\wedge\cdots \wedge P_{e+2}.$$
By the definition of $\Theta_x$, we have
$$\ZZ_p\cdot \Theta_x(\fz_K^{\rm Hg}) = \ZZ_p\cdot \# H^2(\cO_{K,S},T)_{\rm tors} \cdot P_1\wedge\cdots \wedge P_{e+2}, $$
so it is sufficient to prove that
$$\ZZ_p\cdot \# H^2(\cO_{K,S},T)_{\rm tors}  = \ZZ_p \cdot {\rm Eul}_S\cdot \# \sha(E/K)[p^\infty]\cdot {\rm Tam}(E/K) \cdot C_x.$$
It is not difficult to prove this equality, and we leave it as an exercise for the reader. 
\end{proof}


\begin{remark}
In a forthcoming work, we prove that the formula in Theorem \ref{acbsd} implies the conjecture of Bertolini-Darmon (see \cite[Conj.~4.5(1)]{BD} or \cite[Conj.~3.6]{AC}) up to $\ZZ_p^\times$. Furthermore, we formulate a refinement of Conjecture \ref{der iw} so that it essentially implies \cite[Conj. 3.11]{AC}. 
\end{remark}

The following result is an analogue of \cite[Th.~7.6]{bks4} (and a special case of Theorem \ref{descent}). 

\begin{theorem}\label{heegdescent}
Assume Hypothesis \ref{hypk}. If we also assume
\begin{itemize}
\item Conjecture \ref{hpmc} (the Heegner point main conjecture),
\item Conjecture \ref{der iw} for $z_\infty^{\rm Hg}$, and 
\item $R_{K_\infty}^{\rm Boc}\neq 0$, 
\end{itemize}
then the $p$-part of the Birch-Swinnerton-Dyer formula for $E/K$ holds, i.e., 
$$\ZZ_p\cdot L^\ast(E/K,1)=\ZZ_p\cdot \# \sha(E/K)[p^\infty]\cdot {\rm Tam}(E/K)\cdot \frac{1}{\sqrt{|D_K|}} \Omega_{E/K} \cdot R_{E/K}.$$
\end{theorem}

\begin{proof}
This is a consequence of Theorems \ref{descent} and \ref{heeg equivalent}, but we can also deduce it directly by combining Proposition \ref{explicit derivative} and Theorem \ref{acbsd}. 
\end{proof}

\begin{acknowledgments}
The authors would like to thank Francesc Castella for reading an earlier version of this paper and giving them helpful comments.
The first author is supported by JSPS KAKENHI Grant Number 19J00763. The second author is supported by JSPS KAKENHI Grant Number 17K14171.
\end{acknowledgments}

\end{document}